\documentclass[a4paper,10pt,notitlepage,reqno]{article}
\usepackage[english]{babel}
\usepackage{amssymb,amsthm,amsmath} 
\usepackage{mathtools}
\usepackage{mathrsfs}
\usepackage{enumerate}
\usepackage{authblk}
\usepackage[noadjust]{cite}
\usepackage[symbol]{footmisc}
\usepackage{soul}
\usepackage[dvipsnames]{xcolor}
\usepackage{geometry}
\usepackage[colorlinks=true, linkcolor=red, citecolor=blue, urlcolor=blue,hyperfootnotes=false]{hyperref}
\usepackage{cleveref}

\theoremstyle{plain}
\newtheorem{theorem}{Theorem}[section]
\newtheorem{lemma}{Lemma}[section]
\newtheorem{proposition}{Proposition}[section]
\newtheorem{corollary}{Corollary}[section]

\theoremstyle{definition}

\theoremstyle{remark}
\newtheorem{remark}{Remark}[section]

\DeclarePairedDelimiter{\abs}{\lvert}{\rvert} 
\DeclarePairedDelimiter{\norm}{\lVert}{\rVert}

\DeclareMathOperator{\supp}{supp}

\newcommand{\R}{\mathbb{R}}
\newcommand{\C}{\mathbb{C}}

\newcommand{\N}{\mathbb{N}}
\newcommand{\normeq}[1]{{\left\vert\kern-0.25ex\left\vert\kern-0.25ex\left\vert #1 
    \right\vert\kern-0.25ex\right\vert\kern-0.25ex\right\vert}}

\usepackage{braket}
\usepackage{color} 
\newcommand{\Rn}{{\mathbb R^n}}

\newcommand{\alphaAppell}{a}
\newcommand{\betaAppell}{b}

\title{Unique Continuation Properties from one time for hyperbolic Schr\"{o}dinger equations}
\author[1]{Juan Antonio Barcel\'o} 
\author[2]{Biagio Cassano}
\author[3]{Luca Fanelli}		

\affil[1]{ETSI de Caminos, Universidad Polit\'ecnica de Madrid, 28040 Madrid, Spain; juanantonio.barcelo@upm.es}
	
\affil[2]{Dipartimento di Matematica e Fisica, Università degli Studi della Campania ``L. Vanvitelli'', Viale Lincoln 5, 81100 Caserta, Italy; biagio.cassano@unicampania.it}

\affil[3]{Departamento de Matem\'aticas, Universidad del Pa\'is Vasco/Euskal Herriko Unibertsitatea (UPV/EHU), 
Aptdo. 644, 48080, Bilbao, Spain; 
BCAM - Basque Center for Applied Mathematics 48009 Bilbao, Spain;  
Ikerbasque, Basque Foundation for Science, Bilbao, Spain; 
luca.fanelli@ehu.es}

\begin{document}

\date{\small \today}
%%%%%%%%%%%%%%%%%%%%%%%%%%%%%%%%%%%%%%%%%%%%%%%%%%%%%%%%%%%%%%%%%%%%%%%%%%%%%%%%%%%%%%%%%%%%%%%%%%%%%%%%%%%%%%%%%%%%%%%%%%%%%%%%%%%%%%%%%%

\maketitle
%\vspace{-1cm}

%\nocite{*}

%------------abstract------------%

% ---------------------------------%

\begin{abstract}
In this paper, we investigate properties of unique continuation for hyperbolic Schr\"odinger equations with time-dependent complex-valued electric fields and time-independent real magnetic fields. We show that positive masses inside of a bounded region at a single time propagate outside the region and prove gaussian lower bounds for the solutions, provided a suitable average in space-time cylinders is taken.
\end{abstract}
%%%%%%%%%%%%%%%%%%%%%%
%%%%%%%%%%%%%%%%%%%%%%%

\section{Introduction}\label{sec.Introduction}
Let $n\geq2$ and $0< k \leq n$ be integer numbers. For any $v\in\R^n$, let us denote by
\begin{equation}\label{eq:v}
v=:(v_+,v_-)\in\R^k\times\R^{n-k},
\qquad
\widetilde{v}:=(v_+,-v_-) \in \R^n,
\end{equation}
and
consider the 
``hyperbolic'' Laplacian, defined by
$$
L:= \Delta_+-\Delta_-:=\sum_{j=1}^k\frac{\partial^2}{\partial x_j^2}-\sum_{j=k+1}^n\frac{\partial^2}{\partial x_j^2}.
$$
Notice the case $k=n$ reduces to the usual Laplacian.
The time evolutionary Schr\"odinger equation associated to $L$ is the Hyperbolic Schr\"odinger equation
\begin{equation}\label{eq:lilla}
\partial_t u=i(\Delta_+-\Delta_-)u,
\end{equation}
where $u=u(x,t):\R^{n+1}\to\C$.

In this manuscript, we are interested in the study of an electromagnetic perturbation of equation \eqref{eq:lilla} of the form
\begin{equation}\label{eq:main}
  \partial_t u = i (\Delta_{A,+} - \Delta_{A,-}+ V)u,
\end{equation}
where 
\begin{gather*}
u=u(x,t):\R^n \times  [0,1]\to\C,
\\
V=V(x,t)\colon \R^n\times [0,1] \to \C, \\
A = (A^1(x), \dots, A^n(x)) \colon \R^n\to\R^n,
\end{gather*}
 and 
\begin{equation*}
\Delta_{A,+} - \Delta_{A,-}	 = 
\sum_{j = 1}^{k} (\partial_{x_j} - i A_j)^2 
-
\sum_{j = k+ 1}^{n} (\partial_{x_j} - i A_j)^2.
%\begin{pmatrix}
%\partial_{x_1} - i A^1 
%\quad 
%\dots
%\quad 
%\partial_{x_n} - i A^n 
%\end{pmatrix}
%\begin{pmatrix}
%\mathbb{I}_{n_1} & 0 \\
%0 & -\mathbb{I}_{n_2}
%\end{pmatrix}
%\begin{pmatrix}
%\partial_{x_1} - i A^1 
%\\
%\vdots
%\\
%\partial_{x_n} - i A^n 
%\end{pmatrix},
\end{equation*}
The vector field $A$ is interpreted as the \emph{magnetic potential}, while the scalar function $V$ is the \emph{electric potential}.
With the above notation for vectors, we can write $A= (A_+,A_-)\in\R^k\times\R^{n-k}$ and 
   \begin{gather*}
   \nabla_A := \nabla - i A =(\partial_{x_1}-iA^1,\dots,\partial_{x_k}-iA^k,\dots,\partial_{x_n}-iA^n)=: (\nabla_{A,+}, \nabla_{A,-}) \\ 
   \Delta_{A,+} - \Delta_{A,-} = (\nabla_{A,+}\cdot\nabla_{A,+})  - (\nabla_{A,-} \cdot\nabla_{A,-}).
\end{gather*}
The \emph{magnetic field} $B=B(x) : \Rn  \to
M_{n\times n}(\R)$ is the antisymmetric gradient of $A$,
given by
\begin{equation}\label{eq:Bdef}
  B(x)=D_x A(x)-D_x A^t(x),
  \qquad
  B_{jk}(x)=\partial_{x_j} A^k(x)- \partial_{x_k} A^j(x).
\end{equation} 
In dimension $n=3$, where antisymmetric matrices are identified with 3-vectors,  we have
$$
v^tB = \text{curl}A\wedge v,
\qquad(n=3)
$$
for any $v\in\R^3$, 
being the wedge $\wedge$  the vectorial product,
which motivates the usual identification $B=\text{curl}A$.
\vskip0.3cm
Compared with the usual elliptic Schr\"odinger case $(k=n)$, equation \eqref{eq:lilla}, as well as its power-like nonlinear versions, has some peculiar features. First of all notice that the $L^2$--norm is an invariant of the time flow, as in the elliptic case, while at the level of one derivative we have the formal conservation of the \textit{linear energy}, which is an unsigned quantity, given by
\begin{equation}\label{eq:lillalilla}
\int_{\R^n}\left[ |\nabla_+u(x,t)|^2 - |\nabla_-u(x,t)|^2 \right]\,dx,
\end{equation}
over solutions $u$ to \eqref{eq:lilla}. As a consequence, since even the linear energy does not control from above any positive quantity, several basic questions (mostly concerned with the global well posedness) about nonlinear models associated to \eqref{eq:lilla} are completely open, as well as very challenging. As an example, Ghidaglia and Saut proved in \cite{GS} that the nonlinear equation
\begin{equation}\label{eq:Nlilla}
\partial_t u=i(\Delta_+-\Delta_-)u\pm|u|^pu,
\end{equation}
has no traveling wave solutions in $H^1$, independently on the sign of the nonlinear term (due to \eqref{eq:lillalilla}, there is no notion here of {\it focusing} or {\it defocusing}). Among the other nonlinear models based on \eqref{eq:lilla}, we mention the well known {\it Davey-Stewartson} system in $\R^{2+1}$
$$
\begin{cases}
i\partial_tu+\frac{\partial^2u}{\partial x^2}-\frac{\partial^2u}{\partial y^2}=\alpha|u|^2u+\beta u\frac{\partial\varphi}{\partial x}
\\
\Delta\varphi =\frac{\partial}{\partial x}|u|^2,
\end{cases}
$$
where $u,\varphi:\R^{2+1}\to\C,\ u=u(x,y,t),\ \varphi=\varphi(x,y,t)$.
\vskip0.3cm
The $0$-order potential $V$ of equation \eqref{eq:main} is somehow natural from the point of view of nonlinear analysis, since one can always think to the case of a nonlinear function of $u$, namely $V(x,t) = f(u(x,t))$. On the other hand, the magnetic perturbation $A$ in \eqref{eq:main} is quite natural from a geometric point of view, and the hyperbolic nature of the operator $L$ makes matters quite tricky, as we will see in the sequel. In this manuscript we follow a program which we started in \cite{BCF}, in which we investigate about quantitative informations on how the {\it mass} propagates under the evolution flow, in suitable space regions. In particular, in \cite{BCF}, inspired by \cite{agirre2019some}, we are able to describe the mass propagation for equation \eqref{eq:main} with $k=n$ in presence of a suitable magnetic field satisfying some geometric conditions. 
In the present paper, we improve the results in \cite{BCF}, weakening the required assumptions on the magnetic field, and 
we complete the treatment of \eqref{eq:main}, by considering the hyperbolic case $0<k<n$.
\vskip0.3cm
The basic ideas in this topic come from the recent developments on unique continuation properties at two distinct times for solutions to Schr\"odinger equations, in connection with the mathematical manifestations of the Uncertainty Principle from Fourier Analysis. It is well known that, 
if $u(x,0)=O\left(e^{-|x|^2/\beta^2}\right)$ and $u(x,T):=e^{iT\Delta}u(x,0)=
O\left(e^{-|x|^2/\alpha^2}\right)$, then
\begin{align*}
  \alpha\beta<4T
  &
  \Rightarrow
  u\equiv0
  \\
  \alpha\beta=4T
  &
  \Rightarrow
  u(x,0)\ is\ a \ constant\ multiple\ of\ e^{-\left(\frac{1}{\beta^2}+\frac{i}{4T}\right)|x|^2}.
\end{align*}
The corresponding $L^2$-versions of the previous results were proved in \cite{CP} and affirm the following:
\begin{align*}
  e^{|x|^2/\beta^2}f \in L^2,\  e^{4|\xi|^2/\alpha^2}\hat f \in L^2,\
  \alpha\beta\leq4
  &
  \Rightarrow
  f\equiv0
  \\
  e^{|x|^2/\beta^2}u(x,0)\in L^2,\  e^{|x|^2/\alpha^2}e^{iT\Delta}u(x,0) \in L^2,\
  \alpha\beta\leq4T
  &
  \Rightarrow
  u\equiv0.
\end{align*}
We address the reader to \cite{BD, FS, SST, SS} as standard references about this topic.
When the coefficients are not constants, it is usually difficult to involve the Fourier Transform. In the recent years,   
Escauriaza, Kenig, Ponce, and Vega
in the sequel of papers \cite{EKPV0,EKPV1,EKPV2,EKPV3,EKPV4, escauriaza2012uniqueness},
and  with Cowling in \cite{CEKPV} developed purely real analytical
methods to handle the above problems, which permits to obtain sharp
answers for 0-order perturbations of the linear Schr\"odinger equation. Some analogous results have also
been obtained by the authors of the present manuscript in \cite{BFGRV,
  CF1}, in the presence of a non-trivial magnetic fields and in \cite{CF2}, where harmonic oscillators and uniform magnetic fields are considered.
In this direction we refer to \cite{KOR}, where related results are given for the harmonic oscillator, and to  \cite{K}, where the Hardy Uncertainty Principle for general electric second order perturbations is considered. For a treatment in the general elliptic setting with variable coefficients, we refer to \cite{FLY}.
We also mention \cite{FernV, FRRS, JLMP}, where analogous phenomena
are considered for discrete Schr\"{o}dinger evolutions. 
Finally, we refer to the recent survey \cite{FM} for more details and
references to further results.
\vskip0.3cm
\medskip \noindent
The contribution by Agirre and Vega in \cite{agirre2019some}, motivated by the results in \cite{N}, is to answer to similar questions when the decay is assumed at only one time, instead of two. Roughly speaking, they prove the following: if a positive mass is present, for solutions of \eqref{eq:main} with $A\equiv0$ and $V$ bounded, inside of some region (a ball) at one time, then one also observes this mass outside the region, if a suitable time average is taken. This fact can be mathematically translated into a gaussian lower bound for solutions in suitable space-time cylinders. A crucial role in the argument is played by a Carleman estimate, which gives a fundamental bound from below.
\vskip0.3cm
\medskip \noindent
As we saw in \cite{BFGRV, CF1, CF2}, the presence of a magnetic field can produce interesting phenomena in the elliptic case. We are now ready to complete the picture, proving our main result in the hyperbolic case.

\vskip0.3cm
\medskip \noindent

%%%%%%%%%%%%%%%%%%
%%%%%%%%%%%%%%%%%%
%%%%%%%%%%%%%%%%%%
%%%%%%%%%%%%%%%%%%
%%%%%%%%%%%%%%%%%%
%%%%%%%%%%%%%%%%%%
%%%%%%%%%%%%%%%%%%
%%%%%%%%%%%%%%%%%%
%%%%%%%%%%%%%%%%%%
%%%%%%%%%%%%%%%%%%

\begin{theorem}\label{thm:main}
Let $n\geq 2$, $u\in\mathcal{C}([0,1];H^1_{loc}(\mathbb{R}^n))$ be a solution of \eqref{eq:main},
and assume that
\begin{equation}\label{eq:V.bound}
\norm{V}_{L^{\infty}(\mathbb{R}^n\times[0,1])} =: M_V <+\infty,
\end{equation}
\begin{equation}\label{eq:condition.A}
  \int_0^1A(sx)\,ds\in\R^n
  \qquad
  \text{for a.e. }x\in\R^n,
\end{equation}
\begin{equation}\label{eq:decay.B}
 \norm{x^t B}_{L^\infty(\R^n)} + \norm{\widetilde x^tB}_{L^\infty(\R^n)} =: M_B < +\infty,
\end{equation}
being $\widetilde x$ like in \eqref{eq:v}.
Assume moreover that there exists a unit vector $\xi \in \mathbb{S}^{n-1}$ such that 
  \begin{equation}  \label{eq:B.null}
  \norm{\xi^t B}_{L^\infty(\R^n)} =: M_\xi < +\infty.
\end{equation}
Finally assume that there exist $R_0, R_1, M_u>0$,  such that $R_1 > 4(R_0+1)$ and 
\begin{equation}\label{eq:initial.mass}
\int_{\abs{x}\leq R_0} \abs{u(x,0)}^2 dx =: M_u^2,
\end{equation}
\begin{equation}\label{eq:norm.bound}
  \sup_{0\leq t \leq 1} \int_{\abs{x}\leq R_1}
  \left(|u(x,t)|^2+|\nabla_A u(x,t)|^2 \right) dx =: E_u^2 < +\infty.
\end{equation}
Then, there exist $t^\ast = t^\ast(R_0, R_1, M_V, M_B, M_\xi,  M_u, E_u)>0$
and $C = C(M_B) >0$ such that
\begin{equation}\label{eq:thesis}
  M_u^2 \leq C\frac{ e^{C \frac{\rho^2}{t}}}{t}\int_{t/4}^{3t}
  \int_{||y|-\rho-\rho\frac{s}{t}|<4(\rho+1)\sqrt{t}}
  \left(\vert u(y,s)\vert^2 +
  s\vert \nabla_{A} u(y,s) \vert^2\right)\,dyds, \ \ \ \ \  R_0 \leq \rho \leq R_1,
\end{equation}
for any $t\in(0,t^\ast)$.
\end{theorem}

\begin{remark}\label{rem:1}
The assumptions of Theorem \ref{thm:main} are gauge invariant. The condition \eqref{eq:condition.A}  has to be understood as a necessary local integrability condition, in order to get the freedom to choose the so called {\it Cr\"onstrom gauge} (see Section \ref{3sec:cronstrom} below). 
We remark that condition \eqref{eq:condition.A} is not satisfied in the case of homogeneous vector potentials $A$ of degree $-1$, namely the case when the Hamiltonian $\Delta_A$ is scaling invariant. A well known example is given by the {\it Aharonov-Bohm}-type potential $A(x) = \lambda(0,\dots,0,-x_{n},x_{n-1})/(x_{n-1}^2+x_n^2)\in\R^n$, for which the validity of Theorem \ref{thm:main} is still an open question.
\end{remark}
\begin{remark}\label{rem:2}
  The choice of the time interval $[0,1]$ does
  not lead the generality of the results. Indeed, $v\in
  C([0,T],L^2(\Rn))$ 
  is solution to \eqref{eq:main} in \mbox{$\Rn\times[0,T]$}
  if and only if $u\colon \R^n \times [0,1]\to \C$, 
  $u(x,t)=%T^{\frac{n}{4}}
  v(\sqrt{T}x,Tt)$ is solution to
\begin{equation*}
  \partial_t u = i(\Delta_{A_T,+}u - \Delta_{A_T, -}u + V_T (x,t) u)
  \quad \text{ in }\Rn \times [0,1],
\end{equation*} 
where
\begin{equation*} 
A_T(x)=\sqrt{T}A(\sqrt{T}x),\quad V_T(x,t)=T
  V(\sqrt{T}x,Tt).
\end{equation*} 
\end{remark} 
\begin{remark}
Notice that in dimension $n=2$ assumption \eqref{eq:B.null} is equivalent to the requirement $B\in L^\infty(\R^n)$, since the antisymmetric matrix $B$ is identified with a scalar function in this case. Analogously, in dimension $n=2$, assumption \eqref{eq:decay.B} is equivalent to the requirement $|x||B|\in L^{\infty}(\R^n)$, hence we can resume the two above conditions \eqref{eq:decay.B}, \eqref{eq:B.null} in the unified assumption $\langle x\rangle B\in L^\infty(\R^n),\ (n=2)$.  
On the other hand, in higher dimensions $n\geq3$, assumption \eqref{eq:B.null} allows to consider singular magnetic fields: indeed, it is easy to construct unbounded antisymmetric matrices that vanish (and so are bounded) on some fixed directions, cfr. \cite[Remark 1.3]{BFGRV}.
\end{remark}

%
%
%
%========================================
%
%
%
%
%
%
%
%
%
%In the case that $n=2$,  assumption \eqref{eq:B.null} is equivalent to the requirement $B \in L^\infty(\R^n)$, since 
%$B$ is an anti-symmetric $2\times2$-matrix. In fact, in  the general dimensional case, if the magnetic field is bounded it is possible to weaken the assumption \eqref{eq:decay.B}: we complement \Cref{thm:main} with the following proposition.
%\begin{proposition}\label{prop:boundedB}
%Let $n\geq2$, $u\in\mathcal{C}([0,1];H^1_{loc}(\mathbb{R}^n))$ be a solution of \eqref{eq:main},
%and assume \eqref{eq:V.bound}, \eqref{eq:condition.A} and 
%\begin{align}
%\label{eq:xBbounded}
%& \norm{x^t B}_{L^\infty(\R^n)} =:N_{B} < +\infty,
%\\
%\label{eq:Bbounded}
%& \norm{B}_{L^\infty(\R^n)} =: N_{B,\infty} < +\infty.
%\end{align}
%Assume moreover \eqref{eq:initial.mass} and \eqref{eq:norm.bound}. Then \eqref{eq:thesis} holds true.
%\end{proposition}
\begin{remark}
\Cref{thm:main} improves the known results for the usual Schr\"odinger equation
\begin{equation*}
  \partial_t u = i (\Delta_{A} + V)u,
\end{equation*}
(namely the case $k=n$), which has been treated by the authors in the previous work \cite{BCF}. 
Indeed, the analogous result (Theorem 1.1 of \cite{BCF}) is stated under the stronger degeneracy  assumption
\begin{equation*}
\xi^t B (x) = 0 \quad \text{ for a.e. }x\in \R^n,
\end{equation*}
which forces to consider vanishing magnetic fields in dimensions $n=2$. On the contrary, assumption \eqref{eq:B.null} permits us here to prove also a 2D-result, as already remarked above, for bounded magnetic fields.
\end{remark}

As an immediate corollary of the main theorem, we prove the following uniqueness result for small times.
\begin{corollary}
Let $n\geq2$, $u\in\mathcal{C}([0,1];H^1_{loc}(\mathbb{R}^n))$ be a
  solution to \eqref{eq:main}, and let the assumptions of Theorem \eqref{thm:main} be satisfied.
  \begin{itemize}
  \item If there exist $(R_j)_{j \in \N}$, $R_j \to +\infty$ such
      that  for all $j \in \N$
\begin{equation*}
\lim_{t \to 0} C\frac{e^{C \frac{R_j ^2}{t}}}{t} \int_{t/4}^{3t}\int_{||y| - R_j(1+s/t)| < 4(R_j+1)\sqrt{t}} 
\left(|u(y,s)|^2+s\left|\nabla u(y,s)\right|^2\right) dy ds=0,
\end{equation*}
then $u\equiv 0$;
\item  if there
  exists $(t_j)_{j \in \N}\subset(0,t^*)$, $t_j \to 0$ such that for all $j
  \in \N$
  \begin{equation*}
    \lim_{\rho \to +\infty}
    C \frac{e^{C \frac{\rho^2}{t_j}}}{t_j} \int_{t_j/4}^{3t_j}\int_{||y| - \rho(1+s/t_j)| < 4(\rho+1)\sqrt{t_j}} 
  \left(|u(y,s)|^2+s\left|\nabla u(y,s)\right|^2 \right) dy ds=0,
\end{equation*}
then $u\equiv 0$.
\end{itemize}
\end{corollary}
The main ingredient of the proof of Theorem \ref{thm:main} is the following Carleman estimate, that we state here because it is of independent interest.
\begin{proposition}\label{lem:carleman}
  Let $n\geq2$, $R>1$ and $\varphi : [0,1] \to \R$ a smooth function.
  Let 
  \begin{equation*}
  \begin{split}
  & A=A(x,t):\Rn \times [0,1] \to\R^n, \\
  & B:=D_x A-D_x A^t = B(x,t):\Rn \times [0,1] \to \R^{n\times n}
  \end{split}
  \end{equation*}
    and assume that, being $\widetilde x := (x_+, -x_-)$,
\begin{gather}\label{eq:cronstrocond}
%{\color{blue}  x\cdot \partial_t A(x,t) = 0, \quad \text{ for a.e. }x \in \R^n, t \in [0,1], \quad \text{ maybe we don't need it}}\\
\norm{\partial_t A}_{L^{\infty} (\R^n \times [0,1])} < +\infty, \\
  \label{eq:decay.B.carleman}
  \norm{\widetilde x^t B}_{L^\infty(\R^{n}\times [0,1])}   < +\infty.
\end{gather}
Moreover, assume that
  there exists a unitary vector $\xi  = (\xi_+, \xi_-) \in \mathbb{R}^{n}$, $|\xi|=1$, such that 
\begin{equation}
  \label{eq:degenerateB}
%   w\cdot \partial_t A(x,t) = 0, \quad
  \norm{\xi^t B }_{L^\infty(\R^n \times [0,1])} < +\infty.
\end{equation}
Then, being $\widetilde\xi := (\xi_+, -\xi_-)$,
  \begin{equation}\label{eq:Carleman}
    \frac{\tau^{3/2}}{c R^2}
    \norm*{e^{\tau\abs{\frac{x}{R}+\varphi(t) \widetilde{\xi} }^2}g}_{L^2(\R^{n}\times[0,1])}
    \leq
    \norm*{e^{\tau\abs{\frac{x}{R}+\varphi(t) \widetilde{\xi} }^2} (i \partial_t +
      \Delta_{A,+} - \Delta_{A,-})g}_{L^2(\R^{n}\times[0,1])}
  \end{equation}
  for all $g \in
  C_c^{\infty}(\Rn\times [0,1])$ with
  \begin{equation}\label{eq:supp}
    \supp g \subset \left\{ (x,t) \in \R^{n}\times [0,1] :
      \left\vert \frac{x}{R}+\varphi(t) \widetilde{\xi} \right\vert \geq 1  \right\}
  \end{equation}
  and for all $ \tau \geq c R^2 $, being $c \in \R$ such that
  \begin{equation}\label{eq:big.c}
  \begin{split}
  c \geq 
  & \left(
  \norm{\varphi''}_{L^\infty([0,1])} +
   \norm{\varphi'}_{L^\infty([0,1])}^2
  + \norm{\partial_t A}_{L^{\infty}(\R^n \times [0,1])} \right.
  \\
  & \quad   \left.
     +  \norm{\widetilde x^t B}_{L^{\infty}(\R^n \times [0,1])} ^2 
     + \norm{\varphi}_{L^\infty([0,1])}^2 \norm{\xi^t B}_{L^\infty(\R^n \times [0,1])}^2   + 1 \right)^{\frac12}.
     \end{split}
  \end{equation}
\end{proposition}
\begin{remark}\label{rmk:time.A}
In the proof of Theorem \ref{thm:main}, after using the Appell Transformation, we are reduced to an equation with a time-dependent magnetic potential. This motivates the necessity to prove the Carleman estimate \eqref{eq:Carleman} in the general case that $A=A(x,t)$ and $B=B(x,t)$ depend on the time variable.
\end{remark}
\begin{remark}
Remark that, in the assumptions of \Cref{lem:carleman}, only the decay condition $\widetilde{x}^t B \in L^{\infty}(\R^n\times [0,1])$ is required in \eqref{eq:decay.B.carleman}, differently from \eqref{eq:decay.B} in the main \Cref{thm:main}. The additional condition $x^t B \in L^{\infty}(\R^n)$ in the assumption \eqref{eq:decay.B} in \Cref{thm:main} is required in the proof to control the norm $\norm{\partial_t A}_{L^\infty(\R^n\times [0,1])}$ after the use of the Appell transformation in the Cr\"{o}nstrom gauge, see \Cref{rmk:time.A} and \eqref{eq:partialtA} in the proof.
\end{remark}

The rest of the paper is devoted to the proofs of Theorem \ref{thm:main}. and Proposition  \ref{lem:carleman}.

\subsection*{Acknowledgments}
Juan Antonio Barcel\'o acknowledges the support of Ministerio de Ciencia, Innovaci\'on y Universidades of the Spanish goverment through grant PID2021-124195NB-C31.
Biagio Cassano is member of GNAMPA (INDAM) that supports him through the project D86-ALMI22SCROB{\textunderscore}01 ``Dispersion and stability in incompressible fluid dynamics''; also, he has been partially supported by project Vain-Hopes within the program VALERE: VAnviteLli pEr
la RicErca, by the Basque Government through the BERC 2022-2025 program and by the Spanish State Research Agency through BCAM Severo Ochoa excellence accreditation SEV-2017-0718.
L. Fanelli was supported by project PID2021-123034NB-I00 / AEI / 10.13039/501100011033
funded by the Agencia Estatal de Investigaci\'on (Spain), the project IT1615-22 funded
by the Basque Government, and by Ikerbasque.

\section{Preliminaries}
In this preliminary section we prepare some fundamental
tools for the proof of our results.

For the following arguments, it is necessary to generalise  \eqref{eq:Bdef} and to consider a general
time-dependent magnetic potential $A=(A^1,\dots,A^n)\colon \Rn \times
[0,1] \to \Rn$. Consequently, the magnetic field $B : \Rn \times
[0,1]\to \R^{n\times n}$
is given for a.e. $x \in \R^n$ by
\begin{equation}\label{eq:Bdef.time}
  B(x,t)=D_x A(x,t)-D_x A^t(x,t),
  \qquad
  B_{jk}(x,t)=\partial_{x_j} A^k(x,t)- \partial_{x_k} A^j(x,t), \quad j,k=1,\dots,n.
\end{equation} 

\subsection{Cr\"onstrom gauge}\label{3sec:cronstrom}
Equation \eqref{eq:main} is invariant under gauge transformations: for any 
$u$ solution to \eqref{eq:main} and $\varphi=\varphi(x):\R^n\to\R$, 
the function $\widetilde u=e^{-i\varphi}u$ is solution to
\begin{equation*}
  \partial_t\widetilde u= i\left(\Delta_{\widetilde A,+} - \Delta_{\widetilde A,-}+V(x,t)\right) \widetilde u,
\end{equation*}
with
$\widetilde
A=A - \nabla\varphi$. Notice that $\widetilde B := D\widetilde{A} - D\widetilde{A}^t = B$, i.e.~the magnetic field is invariant under gauge transformations.
As in 
\cite{BCF}, we  exploit such gauge invariance and choose here to work in the {\it Cr\"onstrom gauge}  (also called
{\it transversal} or {\it Poincar\'e} gauge), which is given by the following condition
\begin{equation}\label{3eq:cronstrom3}
  x \cdot  \widetilde{A}(x)=0 \quad \text{ for a.e. }x \in  \R^n.
\end{equation} 
If $A$ satisfies \eqref{eq:condition.A}, it is always possible to reduce to the case in which \eqref{3eq:cronstrom3} holds through a gauge transformation, as the following classical result by \cite{I} shows (see also \cite[Lemma 2.2]{BFGRV} for a proof).
\begin{lemma}%[\cite{I}]
  \label{3lem:cronstrom1}
  Let $n\geq2$, $A=A(x)=(A^1,\dots,A^n):\R^n\to\R^n$, 
 $B := DA-DA^t$, and $\Psi(x):=x^tB(x)\in\R^n$ for almost all $x
 \in \R^n$.
Assume that 
\begin{equation}\label{eq:numero}
  \int_0^1\Psi(sx)\,ds\in\R^n,
  \qquad
  \int_0^1A(sx)\,ds\in\R^n,
  \qquad
  \text{for a.e. }x\in\R^n
  \end{equation}
  and denote by
  \begin{gather}
    \label{3eq:cronstrom2}
      \widetilde A (x) := A(x) -\nabla \varphi (x) = - \int_0^1 \Psi(sx)\,ds 
    \\ \label{3eq:varphi}
    \varphi(x):=x\cdot\int_0^1A(sx)\,ds\in\R.
  \end{gather}
  Then $B = D\widetilde A - D \widetilde A^t$, \eqref{3eq:cronstrom3}  holds true and 
  \begin{equation}
 \label{eq:xDA}
 x^t D\widetilde A(x)  = -\Psi(x) +\int_0^1\Psi(sx)\,ds.
 \end{equation}
\end{lemma}
\subsection{Appell Transformation}\label{3sec:Appell}
In this section we generate a family of solutions to \eqref{eq:main}
 by means of the Appell pseudoconformal transformation: we follow the strategy of \cite{BCF}, adapting the transformation to the hyperbolic Laplacian in consideration.
The proof of the following statement is obtained with simple modifications of the arguments in \cite[Lemma 2.7]{BFGRV}, 
so it will not be given here.
\begin{lemma}\label{lem:appell}
  Let $A = (A^1(y,s),\dots ,A^n(y,s)):\Rn \times [0,1] \to \R^{n}$,
\mbox{$V=V(y,s)$}, $F=F(y,s):\R^n \times [0,1] \to\C$, $u=u(y,s):\R^{n}\times[0,1]\to\C$ be a solution to
  \begin{equation}\label{3eq:1appell}
  \partial_s u=i\left((\Delta_{A,+} - \Delta_{A,-})u+V(y,s)u+F(y,s)\right) \quad \text{ in  }
  \Rn \times [0,1]
\end{equation}
and define, for any $\alphaAppell,\betaAppell>0$, the function
\begin{equation}\label{3eq:appell}
  \widetilde u(x,t):=
  \left(\frac{\sqrt{\alphaAppell\betaAppell}}{\alphaAppell(1-t)+\betaAppell t}\right)^{\frac
    n2} \,
  u\left(\frac{x\sqrt{\alphaAppell\betaAppell}}{\alphaAppell(1-t)+\betaAppell
      t},\frac{t\betaAppell}{\alphaAppell(1-t)+\betaAppell t}\right)
  e^{\frac{(\alphaAppell-\betaAppell)}{4i (\alphaAppell(1-t)+\betaAppell t)}(|x_+|^2 - |x_-|^2)}.
\end{equation}
Then $\widetilde u$ is a solution to
\begin{equation}\label{3eq:2appell2}
  \partial_t\widetilde u=i\left((\Delta_{\widetilde A,+} - \Delta_{\widetilde A,-})\widetilde u
    +\frac{(\alphaAppell-\betaAppell)  }{(\alphaAppell(1-t)+\betaAppell t)} 
    (\widetilde A \cdot x)
 \widetilde u+
    \widetilde V(x,t)\widetilde u+\widetilde F(x,t)\right)
  \quad \text{ in } \Rn \times [0,1],
\end{equation}
where
\begin{align}
  \label{3eq:Aappell}
  \widetilde A(x,t)
  &
  = \frac{\sqrt{\alphaAppell\betaAppell}}{\alphaAppell(1-t)+\betaAppell t}
  \, A\left(\frac{x\sqrt{\alphaAppell\betaAppell}}{\alphaAppell(1-t)+\betaAppell t},\frac{t\betaAppell}{\alphaAppell(1-t)+\betaAppell t}\right),
  \\
  \label{3eq:Vappell}
  \widetilde V(x,t)
  &
  = \frac{\alphaAppell\betaAppell}{(\alphaAppell(1-t)+\betaAppell t)^2}
  \, V\left(\frac{x\sqrt{\alphaAppell\betaAppell}}{\alphaAppell(1-t)+\betaAppell t},\frac{t\betaAppell}{\alphaAppell(1-t)+\betaAppell t}\right),
  \\
  \label{3eq:Fappell}
  \widetilde F(x,t)
  &
  = \left(\frac{\sqrt{\alphaAppell\betaAppell}}{\alphaAppell(1-t)+\betaAppell t}\right)^{\frac n2+2}
  \, F\left(\frac{x\sqrt{\alphaAppell\betaAppell}}{\alphaAppell(1-t)+\betaAppell t},\frac{t\betaAppell}{\alphaAppell(1-t)+\betaAppell t}\right)
  e^{\frac{(\alphaAppell-\betaAppell)}{4i
  (\alphaAppell(1-t)+\betaAppell t)}(|x_+|^2 - |x_-|^2)}.
\end{align}
\end{lemma}

\subsection{Carleman estimate}
The main tool in the proof of \Cref{thm:main} is a Carleman estimate for the purely magnetic hyperbolic Schr\"odinger group $i\partial_t+\Delta_{A,+} - \Delta_{A,-}$.
In the following we adapt \cite[Lemma 2.3]{BCF} to consider the hyperbolic Laplacian $\Delta_{A,+} - \Delta_{A,-}$.
The following proposition is needed in the proof of \Cref{thm:main}.
\begin{proof}[Proof of Proposition \ref{lem:carleman}]
Let $g \in C_c^{\infty}(\Rn\times [0,1])$ as in \eqref{eq:supp}.
Denoting $f = e^{\tau\abs{\frac{x}{R}+\varphi(t) \widetilde{\xi}}^2}g$,
an explicit computation shows that
\begin{equation*}
  e^{\tau\abs{\frac{x}{R}+\varphi(t) \widetilde{\xi}}^2}
  (i \partial_t + \Delta_{A,+} - \Delta_{A,-}) g
  =
  \mathcal{S}_\tau f - 4 \tau \mathcal{A}_\tau f, 
\end{equation*}
where $\mathcal{S}_\tau$ and $\mathcal{A}_\tau$ are respectively the symmetric and
anti-symmetric operators
\begin{equation*}
  \begin{split}
    & \mathcal{S}_\tau = i \partial_t + \Delta_{A,+} - \Delta_{A,-} +
    \frac{4\tau^2}{R^2} \left( \left\vert \frac{x_+}{R} + \varphi
      \xi_+\right\vert^2 - \left\vert \frac{x_-}{R} - \varphi \xi_-\right\vert^2\right),
    \\
    & \mathcal{A}_\tau = \frac{1}{R} \left[ \left( \frac{x_+}{R} + \varphi \xi_+ \right)
    \cdot \nabla_{A,+} - \left( \frac{x_-}{R} - \varphi \xi_- \right) \cdot \nabla_{A,-} \right]+ \frac{n_+ - n_-}{2R^2} + \frac{i\varphi'}{2}
    \left(\frac{x}{R} + \varphi \widetilde{\xi} \right)\cdot \widetilde{\xi}.
  \end{split}
\end{equation*}
%\begin{equation*}
%  \begin{split}
%    & \mathcal{S}_\tau = i \partial_t + \Delta_{A,1} - \Delta_{A,2} +
%    \frac{4\tau^2}{R^2} \left( \left\vert \frac{x_1}{R} + \varphi
%      \widetilde{\xi}\right\vert^2 - \left\vert \frac{x_2}{R}\right\vert^2\right),
%    \\
%    & \mathcal{A}_\tau = \frac{1}{R} \left[ \left( \frac{x_1}{R} + \varphi \widetilde{\xi} \right)
%    \cdot \nabla_{A,1} - \frac{x_2}{R} \cdot \nabla_{A,2} \right]+ \frac{n_1 - n_2}{2R^2} + \frac{i\varphi'}{2}
%    \left(\frac{x_1^1}{R} + \varphi \right).
%  \end{split}
%\end{equation*}

We hence have
\begin{equation*}
  \begin{split}
  &\norm{e^{\tau\left\vert \frac{x}{R} + \varphi \widetilde{\xi} \right\vert^2}
      (i\partial_t +\Delta_{A,+} - \Delta_{A,-}) g}_{L^2(\R^{n} \times [0,1])}^2
   \\
   & = \norm{\mathcal{S}_\tau f}_{L^2(\R^{n} \times [0,1])}^2
   + 
   16\tau^2\norm{\mathcal{A}_\tau f}_{L^2(\R^{n} \times [0,1])}^2
    -4\tau \Re
    \langle [\mathcal{S}_\tau,\mathcal{A}_\tau] f ,f \rangle_{L^2(\R^{n} \times [0,1])}
    \\
    & \geq 
    -4\tau \Re
    \langle [\mathcal{S}_\tau,\mathcal{A}_\tau] f ,f \rangle_{L^2(\R^{n} \times [0,1])}.
  \end{split}
\end{equation*}
The explicit computation of $[\mathcal{S}, \mathcal{A}]$ follows easily adapting the computations in \cite[Lemma 4.1]{BFGRV} (see also \cite[Section 2]{FV}), the only relevant change being the computation of 
\begin{equation}
\begin{split}
&\Re \left\langle \left[\Delta_{A,+} - \Delta_{A,-} \, , \, \frac{1}{R}\left(\left(\frac{x_+}{R} + \varphi \xi_+\right) \cdot \nabla_{A,+} - \left(\frac{x_-}{R} - \varphi \xi_-\right) \cdot \nabla_{A,-} \right)\right] f, f \right\rangle_{L^2(\R^{n} \times [0,1])}  
\\ 
& = 
-\frac{2}{R^2} \int_0^1 \int_{\R^n} \left[ \abs{\nabla_{A,+} f}^2 + \abs{\nabla_{A,-} f}^2\right] \, dx dt
\\
& \quad - 
\frac{2}{R} \Im \int_0^1 \int_{\R^n} 
\left(\frac{x}{R} + \varphi \widetilde{\xi}\right)^t 
\begin{pmatrix}
\mathbb{I}_{n_+} & 0 \\
0 & -\mathbb{I}_{n_-}
\end{pmatrix}
B(x)
\begin{pmatrix}
\mathbb{I}_{n_+} & 0 \\
0 & -\mathbb{I}_{n_-}
\end{pmatrix}
\nabla_A f 
\, \overline{f} \, dxdt,
\end{split}
\end{equation}
where $\mathbb{I}_{n_+}$ and $\mathbb{I}_{n_-}$ are the identity matrices of order $n_+$ and $n_-$ respectively.
We conclude that
\begin{equation}\label{eq:use.me} 
  \begin{split}
  &\big\Vert{e^{\tau\left\vert  \frac{x}{R} + \varphi \widetilde{\xi} \right\vert^2}
    (i\partial_t +\Delta_{A,+} - \Delta_{A,-}) g\big\Vert}_{L^2(\R^{n} \times [0,1])}^2
  \\
  & \geq 
    \frac{32 \tau^3}{R^4} \int_0^1 \int_{\R^{n}} \left[ \left\vert \frac{x_+}{R} + \varphi \xi_+
    \right\vert^2 + \left\vert \frac{x_-}{R} - \varphi \xi_-    \right\vert^2 \right]
    \abs{f}^2 \,dx dt
+    \frac{8\tau}{R^2} \int_0^1 \int_{\R^{n}}  \left[ \abs{\nabla_{A,+} f}^2  + \abs{\nabla_{A,-} f}^2 \right]  \, dxdt
  \\
    & \quad + 
    2\tau\int_0^1 \int_{\R^{n}}  \left[ \left( \frac{x}{R} + \varphi \widetilde{\xi} 
      \right)\cdot \widetilde{\xi} \, \varphi''
      + (\varphi')^2 \right] \abs{f}^2 \, dxdt
      \\
      & \quad 
    +\frac{8\tau}{R} \Im \int_0^1 \int_{\R^{n}}  \varphi' (\xi_+ \cdot \nabla_{A,+} + \xi_- \cdot \nabla_{A,-}) f
    \overline{f} \, dxdt
    \\
&   \quad  -\frac{4\tau}{R}\Re \int_0^1 \int_{\R^{n}} 
   \left[ \left(\frac{x_+}{R}+\varphi \xi_+ \right)\cdot (\partial_t A_{+}) - \left( \frac{x_-}{R} - \varphi \xi_-\right) \cdot (\partial_t A_-) \right]  \abs{f}^2 \,dxdt
    \\
    & \quad +
\frac{8\tau}{R} \Im \int_0^1 \int_{\R^n} 
\left(\frac{x}{R} + \varphi \widetilde{\xi}\right)^t 
\begin{pmatrix}
\mathbb{I}_{n_+} & 0 \\
0 & -\mathbb{I}_{n_-}
\end{pmatrix}
B
\begin{pmatrix}
\mathbb{I}_{n_+} & 0 \\
0 & -\mathbb{I}_{n_-}
\end{pmatrix}
\nabla_A f 
\, \overline{f} \, dxdt,
   \\
   & = 
   \normalfont\textsc{I} +
   \normalfont\textsc{II} +
   \normalfont\textsc{III} +
   \normalfont\textsc{IV} +
   \normalfont\textsc{V}.
  \end{split}
\end{equation}
Clearly
\begin{equation}\label{eq:I}
\normalfont\textsc{I} = 
    \frac{32 \tau^3}{R^4} \int_0^1 \int_{\R^{n}}  \left\vert \frac{x}{R} + \varphi \widetilde{\xi}
    \right\vert^2 
    \abs{f}^2 \,dx dt
+    \frac{8\tau}{R^2} \int_0^1 \int_{\R^{n}}   \abs{\nabla_{A} f}^2  \, dxdt.
\end{equation}
Thanks to \eqref{eq:supp},
\begin{equation}\label{eq:II}
\normalfont\textsc{II} \geq 
  -  2\tau  \sup_{t \in [0,1]} [\abs{\varphi''} + \abs{\varphi'}^2]  \int_0^1 \int_{\R^{n}}  \left\vert \frac{x}{R} + \varphi \widetilde{\xi} \right\vert^2
\abs{f}^2 \, dxdt
\end{equation}
and
\begin{equation}\label{eq:III}
\begin{split}
\normalfont\textsc{III} & \geq - \frac{8\tau}{R} 
 \int_0^1 \int_{\R^{n}} 
 \abs{\varphi'} \abs{\xi \cdot \nabla_A f} \abs{f}\, dxdt
 \geq  - \frac{8\tau}{R} 
 \int_0^1 \int_{\R^{n}} 
 \abs{\varphi'} \abs{\nabla_A f} \abs{f}\, dxdt
 \\
 & \geq - {4\tau} \sup_{t\in [0,1]} \abs{\varphi'}^2 
 \int_0^1 \int_{\R^{n}} 
\left\vert \frac{x}{R} + \varphi \widetilde{\xi} \right\vert^2 \abs{f}^2 \,dxdt
- \frac{4\tau}{R^2} 
 \int_0^1 \int_{\R^{n}} 
\abs{\nabla_A f}^2 \,dxdt.
\end{split}
\end{equation}
Thanks to \eqref{eq:supp} and since $R>1$, we have that 
\begin{equation}\label{eq:IV}
\normalfont\textsc{IV} \geq - \frac{8 \tau}{R}  \int_0^1 \int_{\R^{n}} 
\left\vert \frac{x}{R} + \varphi \widetilde{\xi} \right\vert \abs{\partial_t A} \abs{f}^2 \, dxdt
\geq - {8 \tau}%{R} 
\norm{\partial_t A}_{L^{\infty}(\R^n \times [0,1])} 
 \int_0^1 \int_{\R^{n}}  
\left\vert \frac{x}{R} + \varphi \widetilde{\xi} \right\vert^2  \abs{f}^2 \, dxdt.
\end{equation}
To estimate $\normalfont\textsc{V}$, we observe that $\widetilde{\xi} \begin{pmatrix}
\mathbb{I}_{n_+} & 0 \\
0 & -\mathbb{I}_{n_-}
\end{pmatrix} = \xi$. 
%thanks to \eqref{eq:degenerateB} 
%\begin{equation*}
%\norm*{ \widetilde{\xi}^t \begin{pmatrix}
%\mathbb{I}_{n_+} & 0 \\
%0 & -\mathbb{I}_{n_-}
%\end{pmatrix}
%B }_{L^\infty(\R^n \times [0,1])} 
%= 
%\norm{\xi^t B}_{L^\infty(\R^n \times [0,1])}  < +\infty.
%\end{equation*}
Consequently,  % \biagio{ATTENZIONE! qua stai buttando un $R^{-2} \sim 1/\gamma$}
\begin{equation*}
\begin{split}
\normalfont\textsc{V} & \geq 
- \frac{8 \tau}{R^2}  \int_0^1 \int_{\R^{n}} 
\abs{ \widetilde{x}^t B} \abs{\nabla_A f} \abs{f} \,dx dt
- \frac{8 \tau}{R}  \int_0^1 \int_{\R^{n}} 
\abs{\varphi}\abs{ \xi ^t B} \abs{\nabla_A f} \abs{f} \,dx dt
\\
& \geq - \frac{8 \tau}{R^2} \left( 
\norm{ \widetilde{x}^t B }_{L^{\infty}(\R^n \times [0,1])} 
+ \norm{\varphi}_{L^\infty([0,1])} \norm{\xi^t B}_{L^\infty(\R^n \times [0,1])}
\right)
 \int_0^1 \int_{\R^{n}}  \abs{\nabla_A f} \abs{f} \,dx dt
\\
& \geq - \frac{4 \tau } {R^2}
\left(\norm{ \widetilde{x}^t B }_{L^{\infty}(\R^n \times [0,1])}
+ \norm{\varphi}_{L^\infty([0,1])} \norm{\xi^t B}_{L^\infty(\R^n \times [0,1])}\right)^2 
 \int_0^1 \int_{\R^{n}}  \abs{f}^2 \,dxdt 
 \\
 & \quad 
 - \frac{4 \tau}{R^2} 
 \int_0^1 \int_{\R^{n}} \abs{\nabla_A f}^2 \,dxdt.
 \end{split}
 \end{equation*}
Since $R>1$ and thanks to \eqref{eq:supp}, we conclude that 
 \begin{equation}\label{eq:V}
 \begin{split}
 \normalfont\textsc{V} \geq & - {8 \tau }
\left(\norm{ \widetilde{x}^t B }_{L^{\infty}(\R^n \times [0,1])}^2
+ \norm{\varphi}_{L^\infty([0,1])}^2 \norm{\xi^t B}_{L^\infty(\R^n \times [0,1])}^2\right)
 \int_0^1 \int_{\R^{n}} \left\vert \frac{x}{R} + \varphi \widetilde{\xi}\right\vert^2 \abs{f}^2 \,dxdt
 \\ &
 - \frac{4 \tau}{R^2} 
 \int_0^1 \int_{\R^{n}} \abs{\nabla_A f}^2 \,dxdt.
 \end{split}
\end{equation}

From \eqref{eq:use.me}, collecting \eqref{eq:I}--\eqref{eq:V}, we have
\begin{equation*}
\big\Vert{e^{\tau\left\vert \frac{x}{R} + \varphi \widetilde{\xi} \right\vert^2}
  (i\partial_t +\Delta_{A,+} - \Delta_{A,-}) g\big\Vert}_{L^2(\R^{n} \times [0,1])}^2 
\geq \left(   \frac{32 \tau^3}{R^4} - 2
\tau \Psi \right)
  \int_0^1 \int_{\R^{n}}  \left\vert \frac{x}{R} + \varphi \widetilde{\xi} \right\vert^2
\abs{f}^2 \,dx dt,
\end{equation*}
where
\begin{equation*}
\begin{split}
\Psi  := &
\norm{\varphi''}_{L^\infty([0,1])} +
  3 \norm{\varphi'}_{L^\infty([0,1])}^2 
  +4 \norm{\partial_t A}_{L^{\infty}(\R^n \times [0,1])} 
  \\
  & +4  \norm{\widetilde{x}^t B}_{L^{\infty}(\R^n \times [0,1])} ^2
   +4 \norm{\varphi}_{L^\infty([0,1])}^2 \norm{\xi^t B}_{L^\infty(\R^n \times [0,1])}^2.
   \end{split}
\end{equation*}
The thesis is now immediate, since the coefficient at right hand side is bigger than
$\tau^3/c^2 R^4$, if $\tau \geq c R^2$ for $c$ as in \eqref{eq:big.c}.
\end{proof}

\section{Proof of \Cref{thm:main} }
In the following we show \Cref{thm:main}. 
In the proof we follow the strategy of \cite{BCF}, adapting it to consider the hyperbolic Laplacian $\Delta_{A,+} - \Delta_{A,-}$.

  \medskip \noindent
\underline{\bf Reduction to the Cr\"onstrom gauge}. 
We reduce our problem to the Cr\"onstrom gauge, thanks to Lemma \ref{3lem:cronstrom1}. Indeed, in the proof of \Cref{thm:main} assumptions \eqref{eq:condition.A}
 and \eqref{eq:decay.B} 
 give immediately \eqref{eq:numero}.
%, since $\norm{ x^t B }_{L^{\infty}(\R^n)} \leq M_B < +\infty$;
%in the proof of \Cref{prop:boundedB}, \eqref{eq:numero} follows from \eqref{eq:xBbounded}.
Denoting by
  \begin{equation*}
    \varphi(x):= x \cdot\int_0^1 A(sx)\,ds,
    \quad
    \widetilde A(x):=A(x) - \nabla \varphi(x)
\end{equation*}
we have that $B =DA-DA^t = D\widetilde A - D \widetilde
  A^t$, and for a.e. $x \in \R^n$ 
  \begin{gather}
    \label{eq:cronstrom.1}
    \widetilde A(x)= -\int_0^1 (sx)^t B(sx) \,ds,
\\
    x \cdot   \widetilde A(x) = 0.
  \end{gather}
%Moreover, from \eqref{eq:B.null} and \eqref{eq:cronstrom.1} we see that
%  \begin{equation}\label{eq:ventisette}
%    \xi \cdot \widetilde A(x)  =0, \quad \text{ for a.e. }x \in \R^n.
%  \end{equation}
Let $\widetilde u := e^{-i\varphi } u$. Then
(cfr.~\Cref{3sec:cronstrom}) $\widetilde u$ is solution to 
\begin{equation}\label{eq:nuovanuova}
    \partial_t\widetilde u= i\left(\Delta_{\widetilde A,+} - \Delta_{\widetilde A,-}+
      V(x,t) \right)\widetilde u
    \quad \text{ in }\Rn \times [0,1],
\end{equation}
and the conditions in \eqref{eq:initial.mass} and \eqref{eq:norm.bound}
are true replacing $u,A$ with $\widetilde u, \widetilde A$.

\medskip
\medskip \noindent
\underline{{\bf Appell Transformation}}. 
To lighten the notations, in the following we will omit the tildes and just denote
$\widetilde u$, $\widetilde A$
by $u$ and $A$ in \eqref{eq:nuovanuova}. We now apply Lemma \ref{lem:appell} to the equation \eqref{eq:nuovanuova}. 
We choose
  \begin{equation*}
    \alphaAppell, \betaAppell >0, \quad \gamma := \frac{\alphaAppell}{\betaAppell},
  \end{equation*}
  in such a way that
    \begin{equation}
    \label{eq:gamma.big}
    \gamma > \gamma^\ast := \text{max}\left(1, 
    \frac{2}{R_0},
      \frac{ 64  E_u^2 (1+M_V)}{M_u^2}, \frac{4}{R_1 - 4R_0},
      \sqrt{\frac{M_V M_u}{2^{12} E_u}},
      \frac{\sqrt{\abs{2k - n}}}{2^{3/4} R_0},
            \frac{2^{8} E_u}{R_0 M_u}
    \right)^2,
  \end{equation}
Let
\begin{equation}\label{eq:defn.v}
v(x,t) := \alpha(t)^{\frac{n}2}e^{-\frac{i}{4}\beta(t)[\vert x_+ \vert^2 - \vert x_- \vert^2]}u(\alpha(t)x,s(t)), \quad (x,t)\in\mathbb{R}^n\times [0,1],
\end{equation}
with
\begin{equation}\label{eq:change.of.variables}
    \alpha(t)= \frac{1}{(1-t)\sqrt{\gamma}+t/\sqrt{\gamma}}, \quad
    \beta(t)=\frac{1}{1-t+t/\gamma}-\frac{1}{\gamma(1-t)+t},
    \quad
    s(t)=\frac{t}{\gamma(1-t)+t}.
  \end{equation}
    Thanks to \Cref{lem:appell}, $v$ is solution to 
    \begin{equation}\label{3eq:2appell}
      \partial_t v 
      = i\left(\Delta_{\widetilde A,+} - \Delta_{\widetilde A,-} 
        +  \widetilde V(x,t)  \right) v
      \quad \text{ in }\Rn \times [0,1]
    \end{equation}
for $\widetilde A$ and $\widetilde V$  defined by % as in
% \eqref{3eq:Aappell} and \eqref{3eq:Vappell} respectively:
\begin{equation}
  \label{eq:A.V.appell}
    \widetilde A(x,t)
  := \alpha(t) \, A(\alpha(t)x),
  \quad 
  \widetilde V(x,t)
  := (\alpha(t))^2 \,  V(\alpha(t) x , s(t)).
\end{equation}
From \eqref{eq:V.bound}, \eqref{eq:A.V.appell} and since $\sup_{t\in[0,1]} |\alpha(t)| =\sqrt{\gamma}$, we have that
\begin{equation}
\norm{\widetilde V}_{L^{\infty}(\R^n\times[0,1])} = \gamma M_V < +\infty.
\end{equation}
With an explicit computation, \eqref{eq:A.V.appell} gives
\begin{equation*}
\partial_t \widetilde A(x,t) =  (\alpha(t))^2\frac{\gamma -1}{\sqrt{\gamma}} \left[ A(x,t) +
  (\alpha(t) x)^t (D A)(\alpha(t)x)\right].
\end{equation*}
%In the proof of \Cref{thm:main},
Thanks to \eqref{eq:decay.B}, \eqref{eq:xDA}, \eqref{eq:cronstrom.1} and since $\sup_{t\in[0,1]} |\alpha(t)| =\sqrt{\gamma}$, we conclude from above that
\begin{equation}\label{eq:partialtA}
\norm{\partial_t \widetilde A}_{L^\infty(\R^n \times [0,1])} \leq 
3 \gamma^{\frac32} M_B < +\infty.
\end{equation}
%In the proof of \Cref{prop:boundedB}, analogously, thanks to \eqref{eq:xBbounded}, \eqref{eq:xDA} and \eqref{eq:cronstrom.1} 
%we get
%\begin{equation}\label{eq:partialtAbounded}
%\norm{\partial_t \widetilde A}_{L^\infty(\R^n \times [0,1])} \leq 
%3 \gamma^{\frac32} N_B < +\infty.
%\end{equation}
We define the magnetic field $\widetilde B$ according to \eqref{eq:Bdef.time}, getting that  
\begin{equation}\label{eq:tildeB}
\widetilde B(x,t) :=D_x \widetilde A(x,t)-D_x \widetilde A^t(x,t) = (\alpha(t))^2 B(\alpha(t) x) \quad   \text{ for a.e. }x \in \R^n.
\end{equation}
%In the proof of \Cref{prop:boundedB}, thanks to \eqref{eq:Bbounded} and since $\sup_{t\in[0,1]} |\alpha(t)| =\sqrt{\gamma}$,
%we get 
%\begin{equation}\label{eq:2d.bounded.proof}
%\norm{\widetilde B}_{L^\infty(\R^n \times [0,1])} \leq \gamma N_{B,\infty} < +\infty.
%\end{equation}
From \eqref{eq:decay.B} and \eqref{eq:tildeB}, 
we have 
\begin{equation}\label{eq:decay.B.proof}
\norm{(x_+,-x_-)^t \widetilde B(x,t)}_{L^\infty(\R^n \times [0,1])} = 
\sup_{t\in [0,1]}  \sup_{x\in \R^n}  \alpha(t) \left\vert (\alpha(t) x_+,- \alpha(t)x_-)^t B(\alpha(t) x) \right\vert \leq  \sqrt{\gamma} M_B < + \infty.
\end{equation}
From \eqref{eq:B.null}, \eqref{eq:tildeB} and since $\sup_{t\in[0,1]} |\alpha(t)| =\sqrt{\gamma}$, 
we have 
 \begin{equation}\label{eq:ventisettedue}
 \norm{ \xi^t \widetilde B }_{L^\infty(\R^n \times [0,1])} 
 \leq \gamma  \norm{ \xi^t B }_{L^\infty(\R^n \times [0,1])} 
 = \gamma M_\xi < +\infty.
  \end{equation}
We finally remark
that, since $A$ is in the Cr\"onstrom gauge, then $\widetilde A$
is in the Cr\"onstrom gauge too.  
  
  \medskip \noindent
  \underline{{\bf Carleman estimate}}.
Let us denote
\begin{equation}\label{eq:erre}
R := R_0 \sqrt{\gamma},
\end{equation}
so that from \eqref{eq:gamma.big} we
  have $R > 2$. We define the following auxiliary functions:
  \begin{equation*} 
    \theta_R, \eta\in\mathcal{C}^{\infty}([0,+\infty)), \quad
    \theta_R(s)= 
    \begin{cases}
      1 & \text{ if } s\leq R \\
      0 & \text{ if } s\geq R+1,
    \end{cases}
  \quad
 \eta(s)= 
 \begin{cases}
   1 & \text{ if }s\geq 2 \\
   0 & \text{ if }s\leq 3/2,
 \end{cases} 
\end{equation*}
and, with abuse of notation, we denote $\theta_R(x):= \theta_R(\abs{x}), \eta(x):= \eta(\abs{x})$ for all $x \in \R^n$.
Let us assume moreover that for all $s \geq 0$ 
\begin{align}
  \label{eq:control.derivatives.theta}
  & \abs{\theta_R(s)}\leq 1, \quad \abs{\theta_R'(s)} \leq 1,  \quad \abs{ \theta_R''(s)}
    \leq 2, \\
    \label{eq:control.derivatives.eta}
  & \abs{\eta(s)}\leq 1, \quad \abs{\eta'(s)} \leq 2,  \quad \abs{\eta''(s)} \leq
    4.
\end{align}
Also, let
\begin{equation*}
 \varphi\in\mathcal{C}^{\infty}([0,1]), \quad \varphi(t)= 
 \begin{cases}
   4 & \text{ if }t\in [3/8,5/8] \\
   0 & \text{ if }t\in [0,1/4]\cup [3/4,1],
 \end{cases}
\end{equation*}
 such that for all $t \in [0,1]$
\begin{equation}
  \label{eq:control.derivatives.varphi}
  \abs{\varphi(t)}\leq 4, \quad  \abs{\varphi'(t)}\leq 32.
\end{equation}
For $ \xi = (\xi_+, \xi_-)$ in \eqref{eq:B.null}, let $\widetilde{\xi}=(\xi_+, -\xi_-)$, according to the notation in \eqref{eq:v}. 
%In the proof of \Cref{prop:boundedB}, we let $\widetilde{\xi} \in \mathbb{S}^{n-1}$ be an arbitrary vector. 
Set
\begin{equation}\label{eq:defn.g}
g(x,t) := \theta_R(x) \, \eta\left(\frac{x}{R}+\varphi(t)\widetilde{\xi}\right)\, v(x,t),
\quad (x,t)\in\mathbb{R}^n\times [0,1].
\end{equation}
We observe that $\supp g$ is compact and
\begin{equation}\label{eq:supp.g}
  \supp g \subset \left\{(x,t) \in \R^{d}\times[0,1] \, \middle| \,
    \abs{x}\leq R+1, \,
    \frac32 \leq \left|\frac{x}{R}+\varphi(t) \widetilde{\xi} \right|, \,
    t \in \left[\frac14,\frac34\right]
    \right\},
\end{equation}
indeed for $t \in [0,\tfrac14] \cup [\tfrac34,1]$, $g(x,t)$ is non vanishing
if $\frac{3}{2} \leq \frac{\abs{x}}{R}
\leq \frac{R+1}{R}$, that is in contraddiction with $R > 2$ given by \eqref{eq:gamma.big}.

From \eqref{eq:partialtA}, \eqref{eq:decay.B.proof} and \eqref{eq:ventisettedue}, we are in the assumptions of \Cref{lem:carleman}: 
\begin{equation}\label{eq:choose.c}
 c := 
  \left(
  \norm{\varphi''}_{L^\infty([0,1])} +
   \norm{\varphi'}_{L^\infty([0,1])}^2
  + 3\gamma^{\frac32} M_B  +  \gamma M_B^2  + 16 \gamma^2 M_\xi^2 + 1 \right)^{\frac12},
 \end{equation} for all
$\tau \geq c R^2$ it is true that
\begin{equation}\label{eq:carleman.proof}
  \frac{\tau^{3/2}}{cR^2}
  \norm*{e^{\tau\left| \frac{x}{R}+\varphi(t) \widetilde\xi \right|^2}
    g(x,t)}_{L^2(\R^n \times [0,1])} \leq
  \norm*{ e^{\tau\left| \frac{x}{R}+\varphi(t) \widetilde\xi\right|^2}
    (i\partial_t + \Delta_{\widetilde{A},+}-\Delta_{\widetilde{A},-})g(x,t)}_{L^2(\R^n \times [0,1])}.
\end{equation}
%In the proof of \Cref{prop:boundedB}, thanks to \eqref{eq:partialtAbounded}, \eqref{eq:2d.bounded.proof} and \Cref{lem:carleman.bounded}, \eqref{eq:carleman.proof} is true for all $\tau \geq c R^2$, being
%\begin{equation}\label{eq:choose.c}
% c := 2
%  \left(
%  \norm{\varphi''}_{L^\infty([0,1])} +
%   \norm{\varphi'}_{L^\infty([0,1])}^2
%   + 3 \gamma^{\frac32} N_B +
%  + 4\gamma^2 N_{B,\infty}^2 + 1 \right)^{\frac12}.
% \end{equation}
 In the following we estimate from above and from below the
quantities in \eqref{eq:carleman.proof}.
%From now on, the proofs of \Cref{thm:main} and \Cref{prop:boundedB} follow the same path.

\medskip \noindent
\underline{{\bf Estimate from below}}.
We estimate from below the left hand side of
\eqref{eq:carleman.proof}.
Since 
$\left| \frac{x}{R}
  + \varphi(t)\widetilde{\xi}\right| \geq 2$ and $g = \theta_R v$ on $\{\vert
x\vert\leq R+1\}\times [3/8,5/8]$,
we have
\begin{equation}\label{eq:from.below}
  \begin{split}
\norm*{e^{\tau\left| \frac{x}{R}+\varphi(t)\widetilde{\xi}\right|^2}
  g(x,t)}_{L^2(\R^n \times [0,1])}^2
  & = \int_0^1\int_{\R^n} e^{2\tau\left| \frac{x}{R}+\varphi(t)\widetilde{\xi}\right|^2}\vert g(x,t)\vert^2dxdt  \\
  & \geq e^{8\tau}\int_{3/8}^{5/8}\int_{\vert x\vert\leq R+1} \vert \theta_R(x)v(x,t)\vert^2dxdt \\
  & = e^{8\tau}\int_{3/8}^{5/8}\int_{\vert x\vert\leq R+1}
  \alpha(t)^n\vert \theta_R(x)u(\alpha(t)x,s(t))\vert^2dxdt.
\end{split}
\end{equation}
%  & 
We perform the following change of variables in the
integral at right hand side of \eqref{eq:from.below}:
\begin{equation*}
y  = \alpha(t) x, \quad s(t)=\frac{t}{\gamma(1-t)+t}.
\end{equation*}
% Observing that
% \begin{equation*}
% s\left(\left[\frac38,\frac58\right]\right)\subset
% \left[\frac{3}{5\gamma+3}, \frac{5}{3\gamma+5}\right], =:I,
% \end{equation*}
% with a simple computation one sees that
% \begin{equation*}
%   \frac{\gamma}{8}\leq \frac{dt}{ds}(s) \leq \gamma, 
% \end{equation*}
From \eqref{eq:change.of.variables}, we observe that % De hecho es $\alpha(t) \leq \frac{8}{3\sqrt{\gamma}}$
\begin{align}\label{eq:for.change.of.variables.1}
  & \frac{1}{\sqrt{\gamma}} \leq \alpha(t) \leq \frac{3}{\sqrt{\gamma}},
  \quad \text{ for all } t \in \left[\frac38,\frac58\right],
  \\
  \label{eq:for.change.of.variables.2}
  & \frac{\gamma}{8}\leq \frac{dt}{ds}(s) =\frac{\gamma}{(1+s\gamma-s)^2} \leq \gamma,
  \quad \text{ for all }s\in 
  \left[\frac{3}{5\gamma+3}, \frac{5}{3\gamma+5}\right]
  = s\left(\left[\frac38,\frac58\right]\right).
\end{align}
From \eqref{eq:from.below} and \eqref{eq:for.change.of.variables.2}  we conclude that 
\begin{equation*}
  \norm*{e^{\tau\left| \frac{x}{R}+\varphi(t)\widetilde{\xi}\right|^2}
  g(x,t)}_{L^2(\R^n \times [0,1])}^2
\geq e^{8\tau}\frac{\gamma}{8}\int_{\frac{3}{5\gamma+3}}^{\frac{5}{3\gamma+5}}
\int_{\vert y \vert \leq \alpha(t(s))(R+1)}
\left| \theta_R\left(\frac{y}{\alpha(t(s))}\right)\, u(y,s) \right|^2dyds.
\end{equation*}
Clearly then
\begin{equation}\label{eq:carleman.from.below}
  \norm*{e^{\tau\left| \frac{x}{R}+\varphi(t)\widetilde{\xi}\right|^2}
  g(x,t)}_{L^2(\R^n \times [0,1])}^2
\geq e^{8\tau}\frac{\gamma}{8}\int_{\frac{3}{5\gamma+3}}^{\frac{5}{3\gamma+5}}
\int_{\vert y \vert \leq \alpha(t(s))(R+1)}
\left| \theta_R\left(\frac{y}{\alpha(t(s))}\right)\, u(y,0)
\right|^2dyds
+ e^{8\tau}\frac{\gamma}{8} E,
\end{equation}
with
\begin{equation*}
  E = \int_{\frac{3}{5\gamma+3}}^{\frac{5}{3\gamma+5}}
\int_{\vert y \vert \leq \alpha(t(s))(R+1)}
\theta_R^2\left(\frac{y}{\alpha(t(s))}\right)\, \left(\abs{u(y,s)}^2 - \abs{u(y,0)}^2\right) dyds.
\end{equation*}
We estimate from below the first term at right hand side in
\eqref{eq:carleman.from.below}: from the condition $\gamma > 1$ in \eqref{eq:gamma.big}, we have 
$\abs{  [\frac{3}{5\gamma+3}, \frac{5}{3\gamma+5}]
} > 1/(4\gamma)$ and 
thanks to \eqref{eq:initial.mass} and
\eqref{eq:for.change.of.variables.1} we conclude 
\begin{equation}\label{eq:to.conclude.1}
  \begin{split}
& e^{8\tau}\frac{\gamma}{8}\int_{\frac{3}{5\gamma+3}}^{\frac{5}{3\gamma+5}}
\int_{\vert y \vert \leq \alpha(t(s))(R+1)}
\left| \theta_R\left(\frac{y}{\alpha(t(s))}\right)\, u(y,0)
\right|^2dyds
\\ &  \geq
  e^{8\tau}\frac{\gamma}{8}\int_{\frac{3}{5\gamma+3}}^{\frac{5}{3\gamma+5}}
\int_{\vert y \vert \leq \alpha(t(s)) R}
\abs{u(y,0)}^2dyds
 \geq   \frac{e^{8\tau}}{32} \int_{\vert y \vert \leq  R /\sqrt{\gamma}}
\abs{u(y,0)}^2dy
\\ & = \frac{e^{8\tau}}{32}M_u^2.
\end{split}
\end{equation}
To estimate $E$, we observe that from \eqref{eq:main} we get that 
\begin{equation*}
\begin{split}
 \frac{d}{dt}|u|^2  & = 
 -2\Im \left[ \sum_{j=1}^{n_+} \partial_j (u\overline{(\partial_j -i A_j) u}) - 
\sum_{j=n_+ + 1}^{n} \partial_j (u\overline{(\partial_j -i A_j) u}) + V|u|^2   \right]
\\
&  =: -2\Im \left[ \text{div}_+(u\overline{\nabla_{A,+} u}) - \text{div}_-(u\overline{\nabla_{A,-} u}) + V|u|^2   \right] 
\end{split}
\end{equation*}
that gives
\begin{equation*}
\begin{split}
& |u(y,s)|^2-|u(y,0)|^2 
\\
& = -2 \Im \int_0^s \left(\text{div}_+\left(u(y,s')
  \overline{\nabla_{A,+} u(y,s')}\right) - \text{div}_-\left(u(y,s')
  \overline{\nabla_{A,-} u(y,s')}\right) + V(y,s') |u(y,s')|^2\right)ds'.
  \end{split}
\end{equation*}
So 
\begin{equation}\label{eq:E}
\abs{E} \leq E_1 + E_2,
\end{equation}
being
\begin{equation}
\begin{split}
E_1 = 2  \left\vert \int_{\frac{3}{5\gamma+3}}^{\frac{5}{3\gamma+5}} \int_0^s
    \int_{\vert y \vert \leq \alpha(t(s))(R+1)}
    \theta_R^2\left(\frac{y}{\alpha(t(s))}\right)\, 
    \left[ \text{div}_+ \left(u(y,s')
  \overline{\nabla_{A,+} u(y,s')}\right) \right. \right.
  \\ 
   \left. \left. - \text{div}_- \left(u(y,s')
  \overline{\nabla_{A,-} u(y,s')}\right) \right] dyds'ds
   \right\vert
   \end{split}
\end{equation}
and
\begin{equation}
  E_2 = 2 \left\vert \int_{\frac{3}{5\gamma+3}}^{\frac{5}{3\gamma+5}}
     \int_0^s \int_{\vert y \vert \leq \alpha(t(s))(R+1)}
    \theta_R^2\left(\frac{y}{\alpha(t(s))}\right)\,  V(y,s') |u(y,s')|^2 dyds'ds \right\vert.
\end{equation}
To estimate $E_1$ we integrate by parts: there is no boundary
contribution thanks to the choice of $\theta_R$.
Thanks to 
%\eqref{eq:gamma.big} \biagio{porqu\'e lo uso?}, 
\eqref{eq:control.derivatives.theta}
and 
\eqref{eq:for.change.of.variables.1} we have
% {\color{blue} quiero decir que
%   $(\theta_R\nabla\theta_R)\left(\frac{y}{\alpha(t(s))}\right)$ es la
%   funci\'on $(\theta_R\nabla\theta_R)$ computada en la variable
%   $\left(\frac{y}{\alpha(t(s))}\right)$}
\begin{equation}\label{eq:E1}
  \begin{split}
  E_1 & \leq
  4 \left\vert \int_{\frac{3}{5\gamma+3}}^{\frac{5}{3\gamma+5}}
   \frac1{\alpha(t(s))}
    \int_0^s
    \int_{\vert y \vert \leq \alpha(t(s))(R+1)}
    u(y,s') \,
    (\theta_R \nabla_+\theta_R, -\theta_R\nabla_-\theta_R) \left(\frac{y}{\alpha(t(s))}\right)\cdot
    \overline{\nabla_A u(y,s')}  dyds'ds
  \right\vert
  \\
  & \leq 4\sqrt{\gamma}
  \int_{\frac{3}{5\gamma+3}}^{\frac{5}{3\gamma+5}}
  \int_0^s
  \int_{\vert y \vert \leq 4(R+1)/\sqrt{\gamma}}
  \left\vert u(y,s') 
    \overline{\nabla_A u(y,s')} \right\vert dyds'ds
  \\
  & \leq
  4 \sqrt{\gamma} \left(\frac{5}{3\gamma+5} - \frac{3}{5\gamma
      +3}\right)
  \int_0^{\frac{5}{3\gamma+5}}
  \int_{\vert y \vert  \leq 4(R+1)/\sqrt{\gamma}}
  \left\vert u(y,s') 
    \overline{\nabla_A u(y,s')} \right\vert dyds'
  \\
  & \leq
  \frac{8}{\sqrt{\gamma}} \frac{5}{3\gamma+5}
    \sup_{s' \in [0,1]}
    \int_{\vert y \vert  \leq 4(R+1)/\sqrt{\gamma}}
  \left\vert u(y,s') 
    \overline{\nabla_A u(y,s')} \right\vert dy
  \\ 
  & \leq \frac{16}{\gamma^{3/2}}
  \sup_{s' \in [0,1]}
  \int_{\vert y \vert  \leq 4(R+1)/\sqrt{\gamma}}
  \left\vert u(y,s') 
    \overline{\nabla_A u(y,s')} \right\vert dy
    \\
  & \leq \frac{8}{\gamma^{3/2}}
  \sup_{s' \in [0,1]}
  \int_{\vert y \vert  \leq R_1}
  \left( \abs{u(y,s')}^2 +
    \abs*{\nabla_A u(y,s')}^2 \right) dy = \frac{8}{\gamma^{3/2}} E_u^2.
  \end{split}
\end{equation}
To estimate $E_2$, we reason as above and thanks to \eqref{eq:V.bound},
\eqref{eq:norm.bound}
%\biagio{porqu\'e lo uso?} \eqref{eq:gamma.big} 
and \eqref{eq:for.change.of.variables.1}, we get
\begin{equation}\label{eq:E2}
  \begin{split}
  E_2  &  \leq \frac{4 M_V}{\gamma} \int_0^{\frac{5}{3\gamma+5}}
  \int_{\vert y \vert \leq 4(R+1)/\sqrt{\gamma}}
  |u(y,s')|^2 dyds' 
  \\ & \leq \frac{8 M_V}{\gamma^2} \sup_{s' \in [0,1]}
  \int_{\vert y \vert \leq 4(R+1)/\sqrt{\gamma}}
    |u(y,s')|^2 dy
    \leq  \frac{8 M_V E_u^2}{\gamma^2}.
  \end{split}
\end{equation}
Thanks to \eqref{eq:gamma.big}, from \eqref{eq:carleman.from.below}--\eqref{eq:E2}
we conclude that
\begin{equation}\label{eq:from.below.final}
  \norm*{e^{\tau\left| \frac{x}{R}+\varphi(t)\widetilde{\xi}\right|^2}
  g(x,t)}_{L^2(\R^n \times [0,1])}^2
\geq   \frac{e^{8\tau}}{2^6} M_u^2.
\end{equation}

\medskip \noindent
\underline{{\bf Estimate from above}}.
We estimate from above the right hand side of
\eqref{eq:carleman.proof}: from \eqref{3eq:2appell} and \eqref{eq:defn.g}
we have
\begin{equation*}
  (i\partial_t+\Delta_{\widetilde A,+} - \Delta_{\widetilde A, -} )g (x,t) = F_1(x,t)+F_2(x,t)+F_3(x,t) + F_4(x,t),
\end{equation*}
where we have set
\begin{align}
 \label{eq:F1}
F_1(x,t) :=& -\widetilde{V}(x,t) g(x,t), \\
 \notag
F_2(x,t) :=& \,  \theta_R(x) \Big[i\varphi'(t) \widetilde{\xi} \cdot
      \nabla \eta\left(\tfrac{x}{R}+
      \varphi (t) \widetilde{\xi} \right) v(x,t)
      +
      \tfrac{1}{R^{2}} \left(\Delta_{0,+}  \eta \left(\tfrac{x}{R} +
      \varphi (t) \widetilde{\xi} \right) - \Delta_{0,-}  \eta \left(\tfrac{x}{R} +
      \varphi (t) \widetilde{\xi} \right)\right) v(x,t)
      \\
      & \label{eq:F2}
     + \tfrac{2}{R} \left(\nabla_{0,+} \eta\left(\tfrac{x}{R} +
      \varphi (t) \widetilde{\xi} \right), -\nabla_{0,-} \eta\left(\tfrac{x}{R} +
      \varphi (t) \widetilde{\xi} \right) \right)\cdot
      \nabla_{\widetilde A}
      v(x,t) 
      \Big],
%      + \tfrac{2}{R} \nabla_{0,+} \eta\left(\tfrac{x}{R} +
%      \varphi (t) \widetilde{\xi} \right) \cdot
%      \nabla_{\widetilde A,+}
%      v(x,t) - \tfrac{2}{R} \nabla_{0,-} \eta\left(\tfrac{x}{R} +
%      \varphi (t) \widetilde{\xi} \right) \cdot
%      \nabla_{\widetilde A, -}
%      v(x,t) \Big],
      \\ \notag
F_3(x,t):= & \, \eta\left(\tfrac{x}{R}+
      \varphi (t) \widetilde{\xi} \right)
      \Big[(\Delta_{0,+} \theta_R(x) - \Delta_{0,-}\theta_R(x)) v(x,t) 
      \\
      & + 2\nabla_{0,+} \theta_R(x) \cdot
      \nabla_{\widetilde A, +}
      v(x,t) - 2  \nabla_{0,-}\theta_R(x) \cdot
      \nabla_{\widetilde A,-}
      v(x,t)\Big],       
      \\
F_4(x,t) := & \, \left[\tfrac{2}{R} \nabla_{0,+} \theta_R(x) \cdot \nabla_{0,+} \eta \left(\tfrac{x}{R} +
      \varphi (t) \widetilde{\xi} \right)
      - \tfrac{2}{R} \nabla_{0,-} \theta_R(x) \cdot \nabla_{0,-} \eta \left(\tfrac{x}{R} +
      \varphi (t) \widetilde{\xi} \right)\right] v(x,t). 
\end{align}

Consequently,
\begin{equation}\label{eq:from.above}
    \norm*{ e^{\tau\left| \frac{x}{R}+\varphi(t)\widetilde{\xi}\right|^2}
      (i\partial_t + \Delta_{\widetilde{A},+}- \Delta_{\widetilde{A},-})g(x,t)}_{L^2(\R^n \times [0,1])}
    \leq
    \sum_{i=1}^4
        \norm*{ e^{\tau\left| \frac{x}{R}+\varphi(t)\widetilde{\xi}\right|^2}
          F_i(x,t)}_{L^2(\R^n \times [0,1])}.
\end{equation}
We estimate separately the terms at right hand side of the previous inequality.
From \eqref{eq:change.of.variables}, we get that
\begin{equation}\label{eq:for.change.of.variables.3}  
  \frac{1}{\sqrt{\gamma}} \leq \alpha(t) \leq \frac{4}{\sqrt{\gamma}},
  \quad
  0\leq \beta(t) % =\frac{1}{1-t+\gamma^{-1}t}-\frac{1}{\gamma(1-t)+t}
    \leq
  \frac{1}{1-t+\gamma^{-1}t} \leq 4,
  \quad \text{ for all } t \in \left[\frac14,\frac34\right],
\end{equation}
Thanks to \eqref{eq:V.bound}, \eqref{eq:gamma.big}, \eqref{eq:A.V.appell}, \eqref{eq:supp.g}
and \eqref{eq:for.change.of.variables.3}, we have
$\norm{\widetilde{V}}_{L^\infty(\supp g)} \leq
16 M_V/\gamma$ and  
\begin{equation}\label{eq:from.above.1}
  \begin{split}
        \norm*{ e^{\tau\left| \frac{x}{R}+\varphi(t)\widetilde{\xi}\right|^2}
          F_1(x,t)}_{L^2(\R^n \times [0,1])} 
        &\leq \frac{16 M_V}{\gamma}
          \norm*{e^{\tau\left| \frac{x}{R}+\varphi(t)\widetilde{\xi}\right|^2}
            g(x,t)}_{L^2(\R^n \times [0,1])}
        \\ & \leq
          \frac{2^{16} E_u}{M_u}
          \norm*{e^{\tau\left| \frac{x}{R}+\varphi(t)\widetilde{\xi}\right|^2}
            g(x,t)}_{L^2(\R^n \times [0,1])}.
        \end{split}
      \end{equation}
Observe that in the support of $F_2$ we have $\abs*{\frac{x}{R} +
  \varphi(t)\widetilde{\xi}} \leq 2$, and thanks to \eqref{eq:gamma.big}, % $R>2$
\eqref{eq:control.derivatives.theta}--%
% \eqref{eq:control.derivatives.eta},
\eqref{eq:control.derivatives.varphi} and \eqref{eq:supp.g}
we estimate
\begin{equation} \label{eq:F_2}
  \norm*{ e^{\tau\left| \frac{x}{R}+\varphi(t)\widetilde{\xi}\right|^2}
    F_2(x,t)}_{L^2(\R^n \times [0,1])}^2
  \leq 2^{14} e^{8\tau}
  \int_{1/4}^{3/4}
  \int_{\vert x \vert \leq R+1} \left(\abs{v(x,t)}^2
  + \frac{1}{R^{2}} \abs{\nabla_{\widetilde A} v}^2\right)dxdt
  = F_{21} + F_{22}.
\end{equation}
We use again the change of variables in
\eqref{eq:change.of.variables}: we observe that
\begin{equation}
  \label{eq:for.change.of.variables.4}
  \frac{\gamma}{16}\leq \frac{dt}{ds}(s) % =\frac{\gamma}{(1+s\gamma-s)^2} 
  \leq \gamma,
  \quad \text{ for all }s\in 
  \left[\frac{1}{3\gamma+1}, \frac{3}{\gamma+3}\right]
  = s\left(\left[\frac14,\frac34\right]\right).
\end{equation}
Thanks to \eqref{eq:for.change.of.variables.4}, we have that
\begin{equation}\label{eq:join.1}
  F_{21} 
  %  2^{14}e^{8\tau}
  % \int_{1/4}^{3/4}
  % \int_{\vert x \vert \leq R+1} \abs{v(x,t)}^2  dxdt
  \leq 2^{14}e^{8\tau} \gamma
  \int_{\frac{1}{3\gamma+1}}^{\frac{3}{\gamma+3}}
  \int_{\abs{y} \leq \alpha(t(s)) (R+1)}
  \abs{u(y,s)}^2 dy ds.
\end{equation}
Thanks to \eqref{eq:defn.v}, \eqref{eq:for.change.of.variables.3} and \eqref{eq:for.change.of.variables.4},
\begin{equation*}
  \begin{split}
    F_{22} % & = 2^{14}e^{8\tau}\int_{1/4}^{3/4}\int_{\vert x\vert\leq R+1}
    % \frac{1}{R^{2}} \abs{\nabla_{\widetilde A} v(x,t)}^2dxdt \\
    & = {2^{14} e^{8\tau}}\int_{1/4}^{3/4}\int_{\vert x\vert\leq R+1}
    \frac{\alpha(t)^n}{R^{2}} \abs*{\alpha(t)(\nabla_A u)(\alpha(t)x,
      s(t))-\frac{i}{2}\beta(t)\widetilde{x} \, u (\alpha(t)x, s(t))}^2 dxdt \\
    & \leq  2^{14} e^{8\tau}  \gamma
    \int_{\frac{1}{3\gamma+1}}^{\frac{3}{\gamma+3}}
    \int_{|y|\leq \alpha(t(s))(R+1)}
    \frac{1}{R^2}\abs*{\alpha(t(s))\nabla_A u(y,s) -
      \frac{i \beta(t(s)) (y_+, - y_-)}{2\alpha(t(s))} u(y,s)}^2 dyds \\
    & \leq 2^{14} e^{8\tau} \gamma
    \int_{\frac{1}{3\gamma+1}}^{\frac{3}{\gamma+3}}
    \int_{|y|\leq \alpha(t(s))(R+1)}
    \left( \frac{32}{\gamma R^2}|\nabla_A u(y,s)|^2 +
      8\left(1+\frac{1}{R}\right)^2|u(y,s)|^2 \right) dyds.
  \end{split}
\end{equation*}
Thanks to \eqref{eq:gamma.big}, from the last inequality we conclude that
\begin{equation}\label{eq:join.2}
  F_{22} \leq 2^{19} e^{8\tau} \gamma
  \int_{\frac{1}{3\gamma+1}}^{\frac{3}{\gamma+3}}
  \int_{|y|\leq \alpha(t(s))(R+1)}
  \left( |u(y,s)|^2 + |\nabla_A u(y,s)|^2 \right) dyds.
\end{equation}
From \eqref{eq:norm.bound}, \eqref{eq:gamma.big}, % \eqref{eq:for.change.of.variables.3},
\eqref{eq:join.1},
\eqref{eq:join.2} and
since $\abs{[\frac{1}{3\gamma+1},\frac{3}{\gamma+3}]} \leq
\frac{4}{\gamma}$, we get
\begin{equation}\label{eq:from.above.2}
   \norm*{ e^{\tau\left| \frac{x}{R}+\varphi(t)\widetilde{\xi}\right|^2}
    F_2(x,t)}_{L^2(\R^n \times [0,1])}^2
 \leq 2^{22} e^{8\tau}
  \sup_{s \in [0,1]}
  \int_{|y|\leq R_1}
  \left( |u(y,s)|^2 + |\nabla_A u(y,s)|^2 \right) dy
  \leq
  2^{22} e^{8\tau} E_u^2.
\end{equation}
We treat now the term with $F_3$. We observe that in its support we
have $R \leq \abs{x}\leq R+1$ and $|\frac{x}{R}+\varphi(t)\widetilde{\xi}|\leq 6$
thanks to \eqref{eq:control.derivatives.varphi} and since $R>2$, consequence of \eqref{eq:gamma.big}. 
Thanks to \eqref{eq:control.derivatives.theta} and \eqref{eq:control.derivatives.eta} we have
\begin{equation*}
\begin{split}
 & \norm*{ e^{\tau\left| \frac{x}{R}+\varphi(t)\widetilde{\xi}\right|^2}
     F_3(x,t)}_{L^2(\R^n \times [0,1])}^2
    \\ & 
  \leq   
  2 e^{72\tau}\int_{1/4}^{3/4}
  \int_{R \leq \vert x \vert \leq R+1}
\left\vert \Delta_{0,+}\theta_R (x) - \Delta_{0,-} \theta_R (x) \right\vert^2 \abs{v(x,t)}^2
dxdt
\\ & \quad
+   8 e^{72\tau}\int_{1/4}^{3/4}
  \int_{R \leq \vert x \vert \leq R+1}
 \abs{\nabla_{0,+} \theta_R(x) \cdot
      \nabla_{\widetilde A, +}
      v(x,t) - \nabla_{0,-}\theta_R(x) \cdot
      \nabla_{\widetilde A,-}
      v(x,t)}^2 dxdt
      \\
  & =:  F_{31} + F_{32}.
  \end{split}
\end{equation*}
Since
\begin{equation}
\Delta_{0,+}\theta_R (x) - \Delta_{0,-} \theta_R (x) = 
\left( |x_+|^2 - |x_-|^2 \right) \left( \frac{\theta_R''(\abs{x})}{|x|^2} - \frac{\theta_R'(\abs{x})}{|x|^3} \right)
+ (k - (n-k)) \frac{\theta_R'(\abs{x})}{|x|},
\end{equation}
thanks to \eqref{eq:control.derivatives.theta} and since $R>2$, we have
\begin{equation}
\begin{split}
F_{31}(x,t) \leq  & \,
  2 e^{72\tau}\int_{1/4}^{3/4}
  \int_{R \leq \vert x \vert \leq R+1}
 \left\vert  \frac{ |x_+|^2 - |x_-|^2}{R^{3}} ( (R+1) \theta_R''(\abs{x})  - \theta_R'(\abs{x}) )
 +  \frac{2k-n}{R} \theta_R'(\abs{x}) \right\vert^2 \abs{v(x,t)}^2
dxdt \\
\leq
&  \, 16 R^{-2} e^{72\tau}\int_{1/4}^{3/4}
  \int_{R \leq \vert x \vert \leq R+1}
 \left(\left\vert |x_+|^2 - |x_-|^2 \right\vert^2  +  \abs{2k-n}^2 \right) \abs{v(x,t)}^2
dxdt.
\end{split}
\end{equation}
Since $R = \sqrt{\gamma} R_0$, from \eqref{eq:gamma.big} we have that $R>2$ and $|2k - n|^2 < 2^3 R^4$, so we conclude that 
\begin{equation}
\begin{split}
F_{31}(x,t) \leq  & \,
 2^4  e^{72\tau}\int_{1/4}^{3/4}
  \int_{R \leq \vert x \vert \leq R+1}
 \left(2^3 R^2  +  R^{-2}\abs{2k-n}^2 \right) \abs{v(x,t)}^2
dxdt
\\
\leq & \, 2^8 e^{72\tau} R^2 
\int_{1/4}^{3/4}
\int_{R \leq \vert x \vert \leq R+1}
\abs{v(x,t)}^2
dxdt.
\end{split}
\end{equation}

Using the change of coordinates \eqref{eq:change.of.variables}
and reasoning as in the estimate \eqref{eq:join.1}, we have
\begin{equation}\label{eq:F31}
  F_{31}  \leq 
  2^{8} e^{72\tau} \gamma R^{2} 
  \int_{\frac{1}{3\gamma+1}}^{\frac{3}{\gamma+3}}
  \int_{R \leq \frac{\abs{y}}{\alpha(t(s))} \leq  R+1}
 % \left[ (\alpha(t))^{-4} \left\vert |y_+|^2 - |y_-|^2 \right\vert^2  + (n_+ -n_-)^2 \right] 
  \abs{u(y,s)}^2 \, dyds
  =
  2^{8} e^{72\tau} \gamma^2 R_0^{2} 
  \int_{\frac{1}{3\gamma+1}}^{\frac{3}{\gamma+3}}
  \int_{R \leq \frac{\abs{y}}{\alpha(t(s))} \leq  R+1}
%  \left[ \left\vert |y_+|^2 - |y_-|^2 \right\vert^2  + (n_+ -n_-)^2 \right] 
  \abs{u(y,s)}^2 \, dyds.
\end{equation}
To treat $F_{32}$ we observe that
\begin{equation}
\nabla_{\widetilde A, \pm} v (x,t) =
(\alpha(t))^{\frac{n}{2}} e^{-i \frac{\beta(t)}{4}(\abs{x_+}^2 - \abs{x_-}^2)}
\left( \alpha(t) \nabla_{A,\pm} u(\alpha(t) x, s(t)) \mp \frac{i \beta(t)}{2} x_\pm u(\alpha(t)x, s(t)) 
\right),
\end{equation}
so that we have, thanks to \eqref{eq:for.change.of.variables.3} and since $R>2$,
\begin{equation}
\begin{split}
F_{32}(x,t)  \leq &
8 e^{72 \tau} 
\int_{1/4}^{3/4}
  \int_{R \leq \vert x \vert \leq R+1}
  |x|^{-2}
  \abs*{x_+ \cdot \nabla_{\widetilde A, +} v (x,t) - x_- \cdot \nabla_{\widetilde A, -} v(x,t) }^2 \, dxdt
  \\
   \leq  & \,
2^4 e^{72 \tau} 
\int_{1/4}^{3/4}
  \int_{R \leq \vert x \vert \leq R+1}
(\alpha(t))^n |x|^{-2} \big\vert \alpha(t) (x_+ \cdot \nabla_{A, +} u(\alpha(t) x, s(t)) -  x_- \cdot \nabla_{A, -} u(\alpha(t) x, s(t)) )\big\vert^2
\, dxdt
\\
& \, + 
2^4 e^{72 \tau}  
\int_{1/4}^{3/4}
  \int_{R \leq \vert x \vert \leq R+1}
(\alpha(t))^n |x|^{-2} \abs*{\frac{\beta(t)}{2} (\abs{x_+}^2 + \abs{x_-}^2) u(\alpha(t)x,s(t))}^2
\, dxdt
  \\
   \leq  & \,
2^8 e^{72 \tau} 
\int_{1/4}^{3/4}
  \int_{R \leq \vert x \vert \leq R+1}
{\gamma^{-1}}(\alpha(t))^n \abs*{ \tfrac{x_+}{\abs{x}} \cdot \nabla_{A, +} u(\alpha(t) x, s(t)) -  \tfrac{x_-}{\abs{x}} \cdot \nabla_{A, -} u(\alpha(t) x, s(t)) )}^2
\, dxdt
\\
& \, + 
2^8 e^{72 \tau} R^{2}  
\int_{1/4}^{3/4}
  \int_{R \leq \vert x \vert \leq R+1}
(\alpha(t))^n \abs*{u(\alpha(t)x,s(t))}^2
\, dxdt.
\end{split}
\end{equation}
Since $2<R$, $\gamma R^2 = \gamma^2 R_0^2$, thanks to \eqref{eq:defn.v}, \eqref{eq:for.change.of.variables.3},
\eqref{eq:for.change.of.variables.4} and  reasoning as in the estimate
\eqref{eq:join.2} we get
\begin{equation}\label{eq:F32}
\begin{split}
F_{32} \leq & \,
2^8  e^{72 \tau} 
  \int_{\frac{1}{3\gamma+1}}^{\frac{3}{\gamma+3}}
  \int_{R \leq \frac{\abs{y}}{\alpha(t(s))} \leq  R+1}
  \abs*{ \tfrac{y_+}{\abs{y}} \cdot \nabla_{A, +} u(y, s) -  \tfrac{y_-}{\abs{y}} \cdot \nabla_{A, -} u(y, s) )}^2
\, dyds
\\
& \, + 
2^8 e^{72 \tau}\gamma R^{2}  
  \int_{\frac{1}{3\gamma+1}}^{\frac{3}{\gamma+3}}
  \int_{R \leq \frac{\abs{y}}{\alpha(t(s))} \leq  R+1}
\abs*{u(y,s)}^2
\, dyds
\\
\leq & \,
2^6 R^2 \gamma^2 R_0^2  e^{72 \tau} 
\int_{\frac{1}{3\gamma+1}}^{\frac{3}{\gamma+3}}
\int_{R \leq \frac{\abs{y}}{\alpha(t(s))} \leq  R+1}
\frac1{\gamma R^2} \abs*{ \tfrac{y_+}{\abs{y}} \cdot \nabla_{A, +} u(y, s) -  \tfrac{y_-}{\abs{y}} \cdot \nabla_{A, -} u(y, s) )}^2
\, dyds
\\
& \, + 
2^8 e^{72 \tau}\gamma^2 R_0^{2}  
  \int_{\frac{1}{3\gamma+1}}^{\frac{3}{\gamma+3}}
  \int_{R \leq \frac{\abs{y}}{\alpha(t(s))} \leq  R+1}
\abs*{u(y,s)}^2
\, dyds
\\
\leq & \,
2^{8}  e^{72 \tau} \gamma^2 R_0^2 
  \int_{\frac{1}{3\gamma+1}}^{\frac{3}{\gamma+3}}
  \int_{R \leq \frac{\abs{y}}{\alpha(t(s))} \leq  R+1}
\left[ \abs*{u(y,s)}^2 +  \frac1{\gamma} \abs*{ \tfrac{y_+}{\abs{y}} \cdot \nabla_{A, +} u(y, s) -  \tfrac{y_-}{\abs{y}} \cdot \nabla_{A, -} u(y, s) )}^2
 \right]
\, dyds.
\end{split}
\end{equation}
From \eqref{eq:F31} and \eqref{eq:F32}, we have 
\begin{equation*}
\begin{split}
   & \norm*{ e^{\tau\left| \frac{x}{R}+\varphi(t)\widetilde{\xi}\right|^2}
    F_3(x,t)}_{L^2(\R^n \times [0,1])}^2
  \\ & \leq
  2^{9}   e^{72\tau} \gamma^2 R_0^2   \int_{\frac{1}{3\gamma+1}}^{\frac{3}{\gamma+3}}
  \int_{ R\leq \frac{|y|}{\alpha(t(s))} \leq R+1}
  \left[ |u(y,s)|^2+
  \frac1{\gamma} \abs*{ \tfrac{y_+}{\abs{y}} \cdot \nabla_{A, +} u(y, s) -  \tfrac{y_-}{\abs{y}} \cdot \nabla_{A, -} u(y, s) }^2
  \right] dyds.
  \end{split}
\end{equation*}
The length of the above space integration region is $\alpha(t(s))$. In
order to write it in terms of $\gamma$,
we see by \eqref{eq:change.of.variables}, \eqref{eq:for.change.of.variables.3}, and since
$\alpha(t(s))\sqrt{\gamma} = 1 + s(\gamma -1)$, that
\begin{equation*}
\begin{split}
 & \left\{(y,s) \, \middle| \, \alpha(t(s))R\leq |y| \leq \alpha(t(s))(R+1), s\in
    \left[{\frac{1}{3\gamma+1}},{\frac{3}{\gamma+3}}\right]  \right\}
 \\ & \subset
  \left\{ (y,s) \,  \middle| \, \big\vert \abs{y} -R_0 -R_0s\gamma\big\vert
      < \frac{4(R_0+1)}{\sqrt{\gamma}}, s\in
    \left[{\frac{1}{3\gamma+1}},{\frac{3}{\gamma+3}}\right]  \right\},
    \end{split}
  \end{equation*}
since 
     \begin{gather*}
       |y| - R_0 -R_0 s \gamma \geq
      % \alpha(t(s))R - R_0 - R_0 s \gamma
      %=
      % (1 +s\gamma -s) R_0 - R_0(1+s\gamma) = \\
      % & 
       -s R_0 % \geq
      % -\frac{3}{\gamma+3} R_0 \geq -\frac{4}{\sqrt{\gamma}} R_0
       \geq -\frac{4}{\sqrt{\gamma}} (R_0+1), 
\\
       |y| -R_0 -R_0 s \gamma 
   % \leq   \alpha(t(s))R + \alpha(t(s)) - R_0 - R_0 s \gamma
    %   \leq 
     %  \alpha(t(s)) - s R_0 
     \leq \alpha(t(s))
     %  \\ & 
       \leq \frac{1 + s(\gamma -1)}{\sqrt{\gamma}} 
     %  \leq   \frac{1}{\sqrt{\gamma}} +
      % \frac{3}{\sqrt{\gamma}}\frac{\gamma-1}{\gamma+3}
       % \leq \frac{4}{\sqrt{\gamma}} 
       \leq \frac{4}{\sqrt{\gamma}}(R_0+1).
   \end{gather*}
Therefore
\begin{equation}\label{eq:from.above.3}
  \begin{split}
     & \norm*{ e^{\tau\left| \frac{x}{R}+\varphi(t)\widetilde{\xi}\right|^2}
       F_3(x,t)}_{L^2(\R^n \times [0,1])}^2
     \\ & 
\leq  2^{9}   e^{72\tau} \gamma^2 R_0^2
    \int_{\frac{1}{3\gamma+1}}^{\frac{3}{\gamma+3}}
    \int_{||y|-R_0-R_0s\gamma|<\frac{4(R_0+1)}{\sqrt{\gamma}}}
    \left(\abs{u(y,s)}^2 + 
     {\gamma}^{-1} \abs*{ \tfrac{y_+}{\abs{y}} \cdot \nabla_{A, +} u(y, s) -  \tfrac{y_-}{\abs{y}} \cdot \nabla_{A, -} u(y, s)}^2 
    \right) dyds.
  \end{split}
\end{equation}
Finally, we treat the term in $F_4$: we reason analogously as done in
the estimates of the terms in $F_2$ and $F_3$. Thanks to
\eqref{eq:norm.bound}, \eqref{eq:gamma.big},
\eqref{eq:control.derivatives.theta}, 
\eqref{eq:control.derivatives.eta}, \eqref{eq:for.change.of.variables.3}, \eqref{eq:for.change.of.variables.4} and since $R>2$ we have
\begin{equation}\label{eq:from.above.4}
  \begin{split}
    & \norm*{ e^{\tau\left| \frac{x}{R}+\varphi(t)\widetilde{\xi}\right|^2}
      F_4(x,t)}_{L^2(\R^n \times [0,1])}^2
    \\ &
    \leq
       e^{8 \tau} 
    \int_{1/4}^{3/4} \int_{R\leq \abs{x}\leq R+1} \abs{v(x,t)}^2 \, dxdt
    \leq
     e^{8 \tau} \gamma
    \int_{\frac{1}{3\gamma+1}}^{\frac{3}{\gamma+3}}
    \int_{R\leq \frac{\abs{y}}{\alpha(t(s))} \leq R+1}
    \abs{u(y,s)}^2 \,dyds
    \\ &
    \leq
    4 e^{8\tau} \sup_{s\in [0,1]} \int_{\abs{y}\leq R_1}
    \abs{u(y,s)}^2\,dy \leq 4 e^{8\tau} E_u^2.
  \end{split}
\end{equation}

Gathering \eqref{eq:carleman.proof}, \eqref{eq:from.above}, \eqref{eq:from.above.1},
\eqref{eq:from.above.2}, \eqref{eq:from.above.3} and \eqref{eq:from.above.4},
we conclude that 
\begin{equation}\label{eq:from.above.final}
  \begin{split}
  &   \frac{\tau^{3/2}}{cR^2}
  \norm*{e^{\tau\left| \frac{x}{R}+\varphi(t)\widetilde{\xi}\right|^2}
    g(x,t)}_{L^2(\R^n \times [0,1])} \leq
  \norm*{ e^{\tau\left| \frac{x}{R}+\varphi(t)\widetilde{\xi}\right|^2}
    (i\partial_t + \Delta_{\widetilde{A},+} - \Delta_{\widetilde{A},-})g(x,t)}_{L^2(\R^n \times [0,1])}
  \\ & \leq
\frac{2^{16} E_u}{M_u}
       \norm*{e^{\tau\left| \frac{x}{R}+\varphi(t)\widetilde{\xi}\right|^2}
    g(x,t)}_{L^2(\R^n \times [0,1])}
  +
  2^{12} e^{4\tau} E_u
  \\ & \quad +
  \left(2^9   e^{72 \tau} \gamma^2 R_0^2
    \int_{\frac{1}{3\gamma+1}}^{\frac{3}{\gamma+3}}
    \int_{||y|-R_0-R_0s\gamma|<\frac{4(R_0+1)}{\sqrt{\gamma}}}
    \left(\abs{u(y,s)}^2 + 
         {\gamma}^{-1} \abs*{ \tfrac{y_+}{\abs{y}} \cdot \nabla_{A, +} u(y, s) -  \tfrac{y_-}{\abs{y}} \cdot \nabla_{A, -} u(y, s)}^2 
         \right)
    dyds \right)^{\frac{1}{2}}.
\end{split}
\end{equation}

\medskip \noindent
\underline{{\bf Conclusion of the proof}}.
We set 
\begin{equation}\label{eq:defn.tau}
\tau := 64 c R^2.
\end{equation}
Thanks to \eqref{eq:gamma.big} and \eqref{eq:choose.c}, we have 
\begin{equation}\label{eq:cuentas.finales}
   \frac{2^{16}E_u}{M_u} \leq 2^8 \sqrt{\gamma} R_0 \leq 2^8 \sqrt{c} R = \frac{\tau^{3/2}}{2cR^2},
\end{equation}
so from \eqref{eq:from.above.final} we get
\begin{equation*}
  \begin{split}
  &   \frac{\tau^{3/2}}{cR^2}
  \norm*{e^{\tau\left| \frac{x}{R}+\varphi(t)\widetilde{\xi}\right|^2}
    g(x,t)}_{L^2(\R^n \times [0,1])} \leq
    \frac{\tau^{3/2}}{2cR^2}
    \norm*{e^{\tau\left| \frac{x}{R}+\varphi(t)\widetilde{\xi}\right|^2}
    g(x,t)}_{L^2(\R^n \times [0,1])}
  +
  2^{12} e^{4\tau} E_u
  \\ & \quad +
  \left(2^9   e^{72 \tau} \gamma^2 R_0^2
    \int_{\frac{1}{3\gamma+1}}^{\frac{3}{\gamma+3}}
    \int_{||y|-R_0-R_0s\gamma|<\frac{4(R_0+1)}{\sqrt{\gamma}}}
    \left(\abs{u(y,s)}^2 + 
         {\gamma}^{-1} \abs*{ \tfrac{y_+}{\abs{y}} \cdot \nabla_{A, +} u(y, s) -  \tfrac{y_-}{\abs{y}} \cdot \nabla_{A, -} u(y, s)}^2 
         \right)
    dyds \right)^{\frac{1}{2}}.
\end{split}
\end{equation*}
Thanks to \eqref{eq:from.below.final} and \eqref{eq:defn.tau}, this immediately gives
\begin{equation}
  \begin{split}
  \label{eq:1}
  &2^5 e^{4\tau}\sqrt{c}RM_u = \frac{\tau^{3/2}}{2 c R^2} \frac{e^{4\tau}M_u}{2^3}
  \\ & \leq
  2^{12} e^{4\tau} E_u
  +
  \left(2^{10}   e^{72 \tau} \gamma^2 R_0^2
    \int_{\frac{1}{3\gamma+1}}^{\frac{3}{\gamma+3}}
    \int_{||y|-R_0-R_0s\gamma|<\frac{4(R_0+1)}{\sqrt{\gamma}}}
    \left(\abs{u(y,s)}^2 + \gamma^{-1} \abs{\nabla_A u (y,s)}^2\right)
    dyds\right)^{\frac{1}{2}}.
\end{split}
\end{equation}
Thanks to \eqref{eq:choose.c} and \eqref{eq:cuentas.finales}, we have
$  2^{12} e^{4\tau} E_u \leq  2^4 e^{4\tau}RM_u \leq 2^4 e^{4\tau}\sqrt{c}RM_u$ ,
so we conclude that 
\begin{equation*}
  2^8 e^{8\tau} c R^2 M_u^2  \leq
  2^{10}  e^{72 \tau} \gamma^2 R_0^2
    \int_{\frac{1}{3\gamma+1}}^{\frac{3}{\gamma+3}}
    \int_{||y|-R_0-R_0s\gamma|<\frac{4(R_0+1)}{\sqrt{\gamma}}}
    \left(\abs{u(y,s)}^2 + \gamma^{-1} \abs{\nabla_A u (y,s)}^2\right)
    dyds,
\end{equation*}
that is to say
\begin{equation*}
  M_u^2  \leq
  \frac{4   e^{2^{12} c R_0^2 \gamma}}{c}  \gamma
    \int_{\frac{1}{3\gamma+1}}^{\frac{3}{\gamma+3}}
    \int_{||y|-R_0-R_0s\gamma|<\frac{4(R_0+1)}{\sqrt{\gamma}}}
    \left(\abs{u(y,s)}^2 + \gamma^{-1} \abs{\nabla_A u (y,s)}^2\right)
    dyds.
\end{equation*}
Consequently, for $C = \max(2^{12}c, 4 c^{-1}) >0$ we have
\begin{equation}\label{eq:almost.the.end}
  M_u^2  \leq
  C e^{ C R_0^2 \gamma}  \gamma
    \int_{\frac{1}{4\gamma}}^{\frac{3}{\gamma}}
    \int_{||y|-R_0-R_0s\gamma|<\frac{4(R_0+1)}{\sqrt{\gamma}}}
    \left(\abs{u(y,s)}^2 + \gamma^{-1} \abs{\nabla_A u (y,s)}^2\right)
    dyds.
\end{equation}
We let $t:=\gamma^{-1}$: from \eqref{eq:almost.the.end} we get that
\eqref{eq:thesis} holds for all $0< t< t^\ast := (\gamma^\ast)^{-1}$ and
$\rho=R_0$. In order to complete the proof for any $\rho\in[R_0, R_1]$, it is sufficient to repeat the same argument as above, choosing $R=\rho\sqrt\gamma$ in \eqref{eq:erre}.

\bibliographystyle{siam}

\end{document}